\documentclass[letterpaper]{article} 
\usepackage{aaai24}  
\usepackage{times}  
\usepackage{helvet}  
\usepackage{courier}  
\usepackage[hyphens]{url}  
\usepackage{graphicx} 
\usepackage[sort]{natbib}  
\usepackage{caption} 
\usepackage{algorithm}
\usepackage{algorithmic}
\usepackage{microtype}
\usepackage{amsmath}
\usepackage{amssymb}
\usepackage{mathtools}
\usepackage{amsthm}
\usepackage{color}
\usepackage[shortcuts]{extdash}

\usepackage{booktabs}
\usepackage{newfloat}
\usepackage{listings}
\DeclareCaptionStyle{ruled}{labelfont=normalfont,labelsep=colon,strut=off} 
\floatstyle{ruled}
\newfloat{listing}{tb}{lst}{}
\floatname{listing}{Listing}
\theoremstyle{plain}
\newtheorem{theorem}{Theorem}[section]

\newtheorem{lemma}[theorem]{Lemma}

\theoremstyle{definition}
\newtheorem{definition}[theorem]{Definition}
\newtheorem{assumption}[theorem]{Assumption}
\theoremstyle{remark}
\newtheorem{remark}[theorem]{Remark}

\DeclarePairedDelimiter\ceil{\lceil}{\rceil}
\DeclarePairedDelimiter\floor{\lfloor}{\rfloor}
\newcommand\norm[1]{\left\lVert#1\right\rVert}

\def\BR{{\mathbb{R}}}
\def\fO{{\mathcal{O}}}

\usepackage{amsmath}\allowdisplaybreaks
\usepackage{amsfonts,bm}
\usepackage{amssymb}

\def\eqref#1{equation~(\ref{#1})}

\def\ceil#1{\left\lceil #1 \right\rceil}
\def\floor#1{\left\lfloor #1 \right\rfloor}
\def\1{\bf{1}}

\def\fO{{\mathcal{O}}}

\def\fR{{\mathcal{R}}}

\def\BR{{\mathbb{R}}}

\DeclareMathOperator*{\argmax}{arg\,max}
\DeclareMathOperator*{\argmin}{arg\,min}

\usepackage{amsthm}
\theoremstyle{plain}

\makeatletter
\def\Ddots{\mathinner{\mkern1mu\raise\p@
\vbox{\kern7\p@\hbox{.}}\mkern2mu
\raise4\p@\hbox{.}\mkern2mu\raise7\p@\hbox{.}\mkern1mu}}
\makeatother

\makeatletter
\newcommand*{\rom}[1]{\expandafter\@slowromancap\romannumeral #1@}
\makeatother

\title{Incremental Quasi-Newton Methods with Faster Superlinear Convergence Rates}
\author{
    Zhuanghua Liu\textsuperscript{\rm 1, \rm2},
    Luo Luo\thanks{The corresponding author}\textsuperscript{\rm 3},
    Bryan Kian Hsiang Low\textsuperscript{\rm 1}
}
\affiliations{
    \textsuperscript{\rm 1}Department of Computer Science, National University of Singapore\\
    \textsuperscript{\rm 2}CNRS@CREATE LTD, 1 Create Way, \#08-01 CREATE Tower, Singapore 138602\\
    \textsuperscript{\rm 3}School of Data Science, Fudan University\\
    liuzhuanghua9@gmail.com, luoluo@fudan.edu.cn, lowkh@comp.nus.edu.sg
}

\usepackage{bibentry}

\begin{document}

\maketitle

\begin{abstract}
We consider the finite-sum optimization problem, where each component function is strongly convex and has Lipschitz continuous gradient and Hessian.
The recently proposed incremental quasi-Newton method is based on BFGS update and achieves a local superlinear convergence rate that is dependent on the condition number of the problem. 
This paper proposes a more efficient quasi-Newton method by incorporating the symmetric rank-1 update into the incremental framework, which results in the condition-number-free local superlinear convergence rate. 
Furthermore, we can boost our method by applying the block update on the Hessian approximation, which leads to an even faster local convergence rate. 
The numerical experiments show the proposed methods significantly outperform the baseline methods.
\end{abstract}

\section{Introduction}
We study the following finite-sum minimization problem:
\begin{equation}
    \min_{x\in\BR^d} f(x) \coloneqq \frac{1}{n} \sum_{i=1}^{n} f_i (x),
    \label{finite_sum_obj}
\end{equation}
where each individual function $f_i: \BR^d \to \BR$ is strongly convex and has Lipschitz continuous gradient and Hessian. 
This formulation is ubiquitous in various machine learning models, including maximum likelihood estimation (MLE) \cite{bishop2006pattern, bottou2018optimization} and unsupervised learning problems \cite{murphy2012machine, hastie2009elements}. 
A notable example of the problem~(\ref{finite_sum_obj}) is the empirical risk minimization in supervised learning, where $n$ is the number of data examples and $f_i(\cdot)$ corresponds to the loss function incurred by each training instance.

In this paper, we are interested in solving the large-scale finite-sum problem, that is, the number of components $n$ in formulation (\ref{finite_sum_obj}) is large.
In this scenario, accessing the exact gradient or Hessian over the entire dataset is too expensive for each iteration. 
To circumvent this issue, stochastic or incremental optimization methods were introduced since they only require computing an estimation of the gradient or Hessian by a single sample (or a small mini-batch of samples) at each round.
The most popular of these methods is stochastic gradient descent (SGD). 
It has been widely used in large-scale optimization problems thanks to its cheap computational cost per iteration~\cite{bottou2018optimization}.
Applying the variance reduction~\cite{schmidt2017minimizing,johnson2013accelerating, defazio2014saga,zhang2013linear} and acceleration techniques~\cite{nesterov2003introductory,allen2017katyusha} can improve the vanilla SGD, and it achieves a linear convergence rate with optimal incremental first-order oracle complexity~\cite{woodworth2016tight}. 

The second-order methods~\cite{nesterov2003introductory} incorporate the additional curvature information in every iteration, and it is possible to establish the local superliner convergence rate with these methods.
For the finite-sum problem (\ref{finite_sum_obj}), \citet{rodomanov2016superlinearly} proposed the Newton incremental method (NIM), which requires accessing the exact gradient and exact Hessian of one individual function in each iteration and attains the local superlinear convergence rate. 
The classical quasi-Newton methods~\cite{broyden1973local, dennis1974characterization,powell1971convergence} estimate the second-order information 
with first-order oracle calls and still hold the superlinear convergence rate.
However, most of the stochastic variants for quasi-Newton methods \cite{lucchi2015variance, moritz2016linearly} that employ gradient estimators only achieve linear convergence rates. 

\citet{mokhtari2018iqn} proposed the Incremental quasi-Newton (IQN) method by using classical BFGS update~\cite{broyden1970convergence,fletcher1970new,goldfarb1970family,shanno1970conditioning}, which is the first superlinear convergent quasi-Newton method without exact second-order oracle call in each iteration.
However, the best-known analysis of IQN~\cite{mokhtari2018iqn} only provided the asymptotic convergence result. 
Several follow-up works \cite{gao2020incremental, lahoti2023sharpened} attempted to characterize the convergence rate by fusing the greedy quasi-Newton update~\cite{rodomanov2021greedy} into the framework of IQN.
Specifically, \citet{lahoti2023sharpened} proposed sharpened lazy incremental quasi-Newton (SLIQN) by utilizing lazy propagation strategy and showed it achieves the superlinear convergence rate of $\fO\big((1-d^{-1} \varkappa^{-1})^{\ceil{t/n}^2}\big)$, where~$\varkappa$ is the condition number and $t$ is the number of iterations. 
\citet{gao2020incremental} proposed the Incremental Greedy BFGS (IGS) method with the same convergence rate as the SLIQN, but it requires more expensive per-iteration complexity.

In this work, we propose an efficient quasi-Newton method named the Lazy Incremental Symmetric Rank-1 (LISR-1) method for the finite-sum minimization problem. 
Our approach takes advantage of the well-known symmetric rank-1 (SR1) update to construct the Hessian estimator with sharper error bound than BFGS methods, and it also exploits the lazy propagation strategy to maintain a low per-iteration complexity.  
We show that LISR-1 achieves a local superlinear convergence rate of $\fO\big((1-d^{-1})^{\ceil{t/n}^2}\big)$, 
shaving off the dependency on the condition number $\varkappa$ compared with the convergence rate achieved by SLIQN and IGS. 
Each iteration of LISR-1 requires only $O(1)$ incremental gradient/Hessian-vector oracle calls and $\fO(d^2)$ flops in matrix operations, matching the existing IQN methods. 
Furthermore, we extend LISR-1 by making use of the symmetric rank-$k$ update~\cite{liu2023symmetric}  to construct the more accurate Hessian estimator where $k<d$ is the rank of the update, resulting in the block IQN method called Lazy Incremental Symmetric Rank-$k$ (LISR-$k$).
It enjoys the local convergence rate up to~$\fO\big((1- k d^{-1})^{\ceil{t/n}^2}\big)$ with additional computational cost of $\fO(k d^2)$ flops per-iteration. 
The numerical experiments on quadratic programming problems and the model of regularized logistic regression demonstrate significant improvements over baseline methods and confirm our theoretical findings.

\paragraph{Paper Organization} 
In Section \ref{sec:related_works}, we provide a literature review for quasi-Newton methods and their variants for finite-sum optimization problems.
In Section \ref{sec:notations}, we formalize the notations and assumptions of our problem and introduce the background of the Broyden family update. 
In Section \ref{sec:methodology}, we propose our LISR-1 method and provide its convergence analysis. 
In Section \ref{sec:block_update}, we present the LISR-$k$ method by incorporating the block-type update.
In Section \ref{sec:experiment}, we demonstrate the numerical experiments to show the improved efficiency of the proposed methods. 
Finally, we conclude this work in Section~\ref{sec:conclusion}.
All the proofs and more experimental results are deferred to the appendix.

\begin{table*}[t]
\centering
\begin{tabular}{cccc}
    \toprule
    Algorithm & Computation Cost & Convergence Rate \\
    \midrule
    IQN \cite{mokhtari2018iqn}  & $\fO(d^2)$ & asymptotic superlinear  \\\addlinespace
    IGS \cite{gao2020incremental} &  $\fO(d^3)$ & $\fO\big((1-d^{-1} \varkappa^{-1})^{\ceil{t/n}^2}\big)$  \\\addlinespace
    SLIQN \cite{lahoti2023sharpened}  & $\fO(d^2)$ & $\fO\big((1-d^{-1} \varkappa^{-1})^{\ceil{t/n}^2}\big)$  \\\addlinespace 
    \hline \addlinespace
     LISR-1 (this work) & $\fO(d^2)$ & $\fO\big((1-d^{-1})^{\ceil{t/n}^2}\big)$  \\\addlinespace
     LISR-$k$ (this work) & $\fO(k d^2)$ & $\fO\big((1- k d^{-1} )^{\ceil{t/n}^2}\big)$  \\\addlinespace
    \bottomrule
\end{tabular}
\caption{We compare the per-iteration computation cost and the convergence rates of incremental fashion quasi-Newton methods. Note that the explicit convergence rate of the vanilla IQN method still remains a mystery.}
\label{tab:theoretic_res}
\end{table*}

\section{Related Work}\label{sec:related_works} 

In this section, we review related work of quasi-Newton methods and their variants for large-scale optimization problems.

\paragraph{Classical Quasi-Newton Methods} 
Past decades have witnessed extensive research progress on quasi-Newton methods. 
The main advantage of quasi-Newton methods is their capability to reach a superlinear convergence without computing the exact Hessian or its inverse. 
To estimate the second-order information, the classical quasi-Newton methods are based on the secant equation and the corresponding closeness criteria between successive Hessian estimations. 
The choice of closeness criteria leads to different types of quasi-Newton methods, including Broyden's method \cite{broyden1965class,broyden1973local,gay1979some}, the Davidon-Fletcher-Powell (DFP) method \cite{davidon1991variable,fletcher1963rapidly}, the Broyden-Fletcher-Goldfarb-Shanno (BFGS) method \cite{broyden1970convergence, fletcher1970new,goldfarb1970family,shanno1970conditioning} and the symmetric rank-1 (SR1) method \cite{conn1991convergence}. 
The asymptotic superlinear convergence of quasi-Newton methods was established in the 1970s~ \cite{powell1971convergence,dixon1972quasi,dixon1972quasi_2, broyden1973local,dennis1974characterization}, 
while the explicit superlinear rates of quasi-Newton methods were obtained only recently. 
\citet{rodomanov2021greedy} first proposed greedy quasi-Newton methods and gave its non-asymptotic superlinear convergence guarantees. 
Later, \citet{lin2022explicit} provided a sharper analysis for these methods. 
After that, \citet{rodomanov2021new,rodomanov2021rates,jin2023non,ye2021explicit} established the explicit rates for the classical (secant equation-based) quasi-Newton methods.

\paragraph{Block Quasi-Newton Methods} 
\citet{schnabel1983quasi} proposed block quasi-Newton methods. These methods construct the Hessian estimator along multiple directions during each iteration, and they achieve better empirical performance than classical quasi-Newton methods like BFGS~\cite{o1994linear}. 
After several decades, the superlinear convergence of these methods was established by \citet{gao2018block,gower2016stochastic,gower2017randomized}. 
Very recently, \citet{liu2023symmetric} presented explicit superlinear convergence rates of block quasi-Newton methods, which explains why the use of multiple directions benefits the convergence behaviors.

\paragraph{Stochastic/Incremental Quasi-Newton Methods} Due to the sheer volume of data in modern machine learning applications, researchers have been investigating the extension of quasi-Newton methods on large-scale optimization problems.
Several early works established the stochastic quasi-Newton methods to reduce the computational cost at each iteration~\cite{byrd2016stochastic,mokhtari2014res,mokhtari2015global,moritz2016linearly,lucchi2015variance,chang2019accelerated}, but these methods cannot obtain the superlinear convergences like classical quasi-Newton methods. 
Incremental quasi-Newton methods (IQN)~\cite{mokhtari2018iqn,gao2020incremental,lahoti2023sharpened} use the aggregated information to construct a more accurate gradient and Hessian estimator, which leads to superlinear convergence. 
We compare the proposed methods with related work in Table~\ref{tab:theoretic_res}.

\section{Preliminaries}\label{sec:notations}
In this section, we formalize the notations and assumptions throughout this paper, then we introduce the well-known Broyden family updates which are widely used in quasi-Newton methods.

\subsection{Notations} 
We denote $e_i \in \BR^d$ as the $i$-th standard basis vector of~$d$-dimensional Euclidean space, where $i\in[d]$. 
We define the index $i_t$ as $t~{\rm mod}~n$. 
For vectors $u, v \in \BR^d$, we denote their inner product by $\langle u, v \rangle \coloneqq u^{\top} v$.
We use $\norm{\cdot}$ to represent the Euclidean norm of the vector and the spectral norm of the matrix. 
Given a positive semi-definite matrix $A\in\BR^{d\times d}$ and a vector $u\in\BR^d$, we define the norm of $u$ with respect to $A$ as~$\norm{u}_{A} \coloneqq \sqrt{\langle u, A u\rangle}$.
We let
\begin{equation}
    E_k(A) = [e_{i_1} ; \ldots; e_{i_k}] \in \BR^{d\times k},
    \label{largest_k_def}
\end{equation}
where $i_1, \ldots, i_k$ are the indices for the largest $k$ entries in the diagonal of $A$.
We also use ${\rm tr}(\cdot)$ to present the trace of a square matrix.
Additionally, we denote the solution of problem (\ref{finite_sum_obj}) as $x^*\coloneqq\argmin_{x\in\BR^d}f(x)$.

\subsection{Assumptions}

In the remainder of this paper, we always suppose Problem~(\ref{finite_sum_obj}) satisfies the following assumptions.

\begin{assumption}\label{convex_assumption}
We suppose each function $f_i(\cdot)$ is twice-differentiable, $L$-smooth and $\mu$-strongly convex, i.e., there exist constants $L > 0$ and $\mu > 0$ such that 
\begin{equation}\label{eq:strong_convex}
    \mu I \preceq \nabla^2 f_i (x) \preceq L I  
\end{equation}
for any $x\in\BR^d$.
\end{assumption}

\begin{assumption}\label{smooth_assumption}
    We suppose each $f_i(\cdot)$ has a $\Tilde{L}$-Lipschitz continuous Hessian, i.e., there exists a constant $\Tilde{L}$ such that 
    \begin{equation*}
        \norm{\nabla^2 f_i(x) - \nabla^2 f_i(y)} \leq \Tilde{L} \norm{x - y}.
    \end{equation*}
    for any $x, y \in \BR^d$.
\end{assumption}

The strong convexity and the Lipschitz continuity of Hessian in our assumptions imply that each $f_i(\cdot)$ is strongly self-concordant with constant $M \coloneqq\Tilde{L} \mu^{-3/2}$~\cite{rodomanov2021greedy}, i.e, we have
\begin{align*}
\nabla^2 f_i(y) - \nabla^2 f_i(x) \preceq M \norm{y - x}_{\nabla^2 f_i(z)} \nabla^2 f_i(w)    
\end{align*}
for any $x,y,z,w \in \BR^d$.

Additionally, we let $\varkappa\coloneqq L/\mu$ be the condition number of our problem which could be very large in practice.

\subsection{Broyden Family Update}

Many popular quasi-Newton methods such as DFP, BFGS, and SR1 belong to the Broyden family update~\cite[Section 6.3]{nocedal1999numerical}, which is defined as follows.


\begin{definition}
    Let $G\in\BR^{d \times d}$ and $A\in\BR^{d \times d}$ be two positive define matrices satisfying $G \succeq A$. For any non-zero $u \in \BR^d$ and $\tau\in[0,1]$, if $G u = A u$, we define 
    ${\rm Broyd}_{\tau}(G, A, u) \coloneqq G$. Otherwise, we define
    \begin{equation}
    \small
    \begin{split}
        \!& \!{\rm Broyd}_{\tau}(G, A, u) \\ 
        \!\!\coloneqq & \tau \left[ G - \frac{A u u^{\top}G + G u u^{\top} A}{u^{\top} A u} + \left( \frac{u^{\top} G u}{u^{\top} A u} + 1\right) \frac{A u u^{\top} A}{u^{\top} A u} \right] \\
        & \quad + (1 - \tau) \left[ G - \frac{(G - A) u u^{\top} (G - A)}{u^{\top} (G - A) u}\right].
    \end{split}
    \label{broyden_eq}
    \end{equation}
\end{definition}

We can recover several well-known quasi-Newton methods by taking the different values of $\tau$:
\begin{itemize}
\item  For $\tau = 1$, Eq. (\ref{broyden_eq}) corresponds to the DFP update
\begin{equation*}
\begin{split}
    {\rm DFP}(G, A, u) \coloneqq & G - \frac{A u u^{\top} G + G u u^{\top} A}{u^{\top} A u}  \\
    & \quad + \left( \frac{u^{\top} G u}{u^{\top} A u} + 1 \right) \frac{A u u^{\top} A}{u^{\top} A u}.
\end{split}
\end{equation*}
\item For $\tau = \dfrac{u^{\top} A u}{u^{\top} G u}$, we recover the BFGS update
\begin{equation*}
    {\rm BFGS}(G, A, u) \coloneqq G - \frac{G u u^{\top} G}{u^{\top} G u} + \frac{A u u^{\top} A}{u^{\top} A u}.
\end{equation*}
\item For $\tau=0$, we achieve the SR1 update
\begin{equation}
    {\rm SR1}(G, A, u) \coloneqq G - \frac{(G - A) u u^{\top} (G - A)}{u^{\top} (G - A) u}.
    \label{sr1_def}
\end{equation}
\end{itemize}


We can generalize the Broyden family updates with multiple directions \cite{gao2018block,liu2023symmetric,gower2016stochastic,gower2017randomized}. 
In particular, \citet{liu2023symmetric} establish the block version of the SR1 update called the symmetric rank-$k$ (SR-$k$) update, which is defined as follows.

\begin{definition}\label{dfn:srk}
Let $A \in \BR^{d \times d}$ and $G \in \BR^{d \times d}$ be two positive-definite matrices satisfying $G \succeq A$. 
For any full rank matrix $U\in\BR^{d\times k}$ with $k < d$, we define $\textrm{SR-}k (G, A, U)  \coloneqq G$ if $G U = A U$. Otherwise, we define
\begin{equation*}
\small\begin{split}
        \textrm{SR-}k (G, A, U)\coloneqq G - (G\!-\!A) U (U^{\top} (G\!-\!A) U)^{\dag} U^{\top} (G\!-\!A).
\end{split}
\end{equation*}

\begin{remark}    
Note that the SR-$k$ update shown in the above definition is equivalent to the SR1 update when $k=1$.
\end{remark}
\end{definition}

\section{Methodology}\label{sec:methodology}

In this section, we propose the lazy incremental symmetric rank\Hyphdash*1 (LISR-1) method and provide theoretical analysis to show it enjoys condition number-free local superlinear convergence.



\subsection{The Algorithm}

\begin{algorithm}[tb]
\caption{LISR-1}\label{alg:iqn_algos}
\begin{algorithmic}[1] 
\STATE \textbf{Input}: $x^0\in\BR^d$ and $\{B_i^0\in\BR^{d\times d}\}_{i=1}^n$.\\[0.1cm]
\STATE Initialize $t=0$ and $z_i^0 = x^0$ for any $i \in [n]$. \\[0.1cm]
\STATE \textbf{while} not converged \textbf{do} \\[0.1cm]
\STATE \quad Update $x^{t+1}$ as per (\ref{cur_iter_update}). \\[0.1cm]
\STATE \quad Update $z_{i}^{t+1}$ as per (\ref{active_idx_update}). \\[0.1cm]
\STATE \quad Update $B_{i}^{t+1}$ as per (\ref{greedy_sr1_update-0})-(\ref{greedy_sr1_omega}). \\[0.1cm]
\STATE \quad Increment the iteration counter $t$. \\[0.1cm]
\STATE \textbf{end while} \\[0.1cm]
\STATE \textbf{Output}:  $x^t$.
\end{algorithmic}
\end{algorithm}

We first introduce the main intuitions of LISR-1.
For each component function $f_i(x)$, we consider its quadratic approximation at point $z_i^t\in\BR^d$ as
\begin{align*}
\small\begin{split}    
f_i(x) \approx & {\tilde f}_i^t(x) \\
= & f_i (z_i^t) + \nabla f_i (z_i^t)^{\top} (x - z_i^t) + \frac{1}{2}(x- z_i^t)^{\top} B_i^t (x - z_i^t),
\end{split}
\end{align*}
where we estimate $\nabla^2 f_i(z_i^t)$ by a positive-definite matrix $B_i^t\in\BR^{d\times d}$.
Then we obtain $x^{t+1}$ by minimizing the average of~$\{\tilde f_i^t(x)\}_{i=1}^n$, which has the closed form solution
\begin{align}\label{cur_iter_update}
\begin{split}    
    x^{t+1} =& \argmin_{x\in\BR^d} \frac{1}{n}\sum_{i=1}^n\tilde f_i^t (x) \\
    =& \left( \sum_{i=1}^n B_{i}^{t} \right)^{-1} \left( \sum_{i=1}^n B_{i}^{t} z_{i}^{t} - \sum_{i=1}^n \nabla f_i(z_i^{t})\right).
\end{split}
\end{align}
We only update one of $\{z_i^t\}_{i=1}^n$ at each iteration
in a cyclic fashion to make the algorithm efficient, that is
\begin{align}\label{active_idx_update}
    z_i^{t+1} = \begin{cases}
        x^{t+1},  & \text{if}~ i=i_t, \\[0.1cm]
        z_i^{t},  & \text{otherwise},
    \end{cases}
\end{align}
where $i_t=t~{\rm mod}~n$ is the index of the component
we choose at the $t$-th iteration.

We also wish to construct the Hessian estimators efficiently and keep a fast convergence rate.
In particular, we introduce the scaling parameter $\omega^{t+1}$ and apply the SR1 update on one of the individual Hessian estimators in each iteration:
\begin{itemize}
\item For $i=i_t$, we let
\begin{equation}\label{greedy_sr1_update-0}
\small\begin{split}    
    \!\!\!B_{i}^{t+1} = \omega^{t+1} \textrm{SR1}(B_{i}^{t}, \nabla^2 f_{i}(z_{i}^{t+1}), \Bar{u}(B_{i}^{t}, \nabla^2 f_{i}(z_{i}^{t+1}))), 
\end{split}
\end{equation}
where $\Bar u(\cdot,\cdot)$ is the greedy direction which is defined as
\begin{align}\label{sr_greedy_vec}    
\Bar{u}(G, A) 
\coloneqq\argmax_{u \in \{ e_i \}_{i=1}^d } u^{\top} (G-A) u.
\end{align}
\item For $i\neq i_t$, we let
\begin{equation}\label{greedy_sr1_update-1}
\begin{split}    
    B_{i}^{t+1} = \omega^{t+1} B_{i}^{t}. 
\end{split}
\end{equation}
\end{itemize}
Additionally, we set 
\begin{align}\label{greedy_sr1_omega}
\omega^t = \begin{cases}
\big(1 + M \sqrt{L}r_0 \rho^{\ceil{\frac{t}{n}}}\big)^2,   & \text{if}~n~{\rm mod}~t = 0, \\[0.1cm] 
1, & \text{otherwise}, 
\end{cases}    
\end{align}
for some $\rho\in(0,1-d^{-1})$ and let $r_0$ be an upper bound of $\norm{x_0-x^*}$.
This setting implies the step of scaling is executed once every $n$ iterations. 




We present the whole procedure of the proposed LISR-1 in Algorithm~\ref{alg:iqn_algos}.
We can verify that the per-iteration cost of our algorithm is~$\fO(d^2)$ flops.
Notice that the main cost of LISR-1 comes from the computation of Eq. (\ref{cur_iter_update}), which is dominated by maintaining the inverse of the following sum of individual Hessian estimators
\begin{equation*}
    \Bar{B}^{t+1} \coloneqq \sum_{i=1}^n B_i^{t+1}.
\end{equation*}
We can rewrite the above matrix in the recursive form as
\begin{equation}\label{sum_of_hessian}
\Bar{B}^{t+1} = \Bar{B}^{t} + B_{i_t}^{t+1} - B_{i_t}^{t}.
\end{equation}
In the case of $t~{\rm mod}~n \neq 0$, no scaling is performed since we have $\omega^{t+1}=1$. Denote $\Bar{u}^t$ as the abbreviation of $\bar{u}(B_{i_t}^t, \nabla^2 f_{i_t}(z_{i_t}^t))$, then applying the Sherman-Morrison formula on Eq. (\ref{sum_of_hessian}) implies
\begin{equation}\label{eq:smf}
    (\Bar{B}^{t+1})^{-1} = (\Bar{B}^{t})^{-1} + \frac{(\Bar{B}^{t})^{-1}v^{t}(v^{t})^{\top}(\Bar{B}^{t})^{-1}}{(\bar{u}^t)^{\top} v^{t} - (v^{t})^{\top} (\Bar{B}^{t})^{-1} v^{t}},
\end{equation}
where $v^{t}$ is defined as
\begin{equation*}
    v^{t} = (B_{i_{t}}^{t} - \nabla^2 f_{i_{t}} (z_{i_{t}}^{t+1})) \Bar{u}^t.
\end{equation*}
It is easy to observe that computing the right-hand side of Eq.~(\ref{eq:smf}) takes $\fO(d^2)$ flops for given $(\Bar B^t)^{-1}$ and $v^t$.
In the case of~$t~{\rm mod}~n = 0$, each Hessian estimator may be scaled by a factor $\omega^t\neq 1$, which results in the additional computational cost of $\fO(n d^2)$ flops. However, the amortized per-iteration complexity of this step is still~$\fO(d^2)$ because the scaling occurs once per $n$ iterations.
We provide a more efficient implementation of LISR-1 in the appendix.

\subsection{Convergence Analysis} 

We analyze the convergence of LISR-1 by considering the Euclidean distance to the optimal solution $x^*$.
Firstly, the formula (\ref{cur_iter_update}) indicates the general result:

\begin{lemma}\label{lemma:cur_iter_bound}
    The iteration formula (\ref{cur_iter_update}) satisfies
    \begin{equation}
    \small
    \begin{split}
        \norm{x^{t+1} - x^*} \leq& \frac{\Tilde{L} \Gamma^{t}}{2} \sum_{i=1}^n \norm{z_i^{t} - x^*}^2  \\
        & + \Gamma^{t} \sum_{i=1}^n \norm{B_{i}^{t} - \nabla^2 f_i (z_i^{t})} \norm{z_i^{t} - x^*},
    \end{split}
    \end{equation}
    for all $t \geq 1$, where $\Gamma^t \coloneqq \big\|\big(\sum_{i=1}^n B_i^t\big)^{-1}\big\|$. 
\end{lemma}
\begin{remark}
Notice that the proof of Lemma \ref{lemma:cur_iter_bound} only requires the Lipschitz continuity of each $\nabla^2 f_i(\cdot)$ and the iteration formula (\ref{cur_iter_update}).
The validity of this lemma does not rely on the specific choice of Hessian estimators $\{B_i^t\}_{i=1}^n$ and it also can be used to analyze the other incremental fashion methods~\cite{mokhtari2018iqn,lahoti2023sharpened,gao2020incremental}.
\end{remark}

In view of Lemma \ref{lemma:cur_iter_bound}, the more accurate Hessian estimator $B_i^t\approx \nabla^2 f_i(z_i^t)$ can lead to the tighter upper bound of $\norm{x^{t+1}-x^*}$.
Hence, the key to showing the advantage of the proposed method is bounding the difference between $B_i^t$ and~$\nabla^2 f_i(z_i^t)$.
In particular, we introduce the quantity
\begin{equation}
    \nu(G, A) \coloneqq \frac{d \varkappa {\rm tr}(G-A)}{{\rm tr}(A)},
    \label{nudef}
\end{equation}
to describe the difference between two positive definite matrices $G\in\BR^{d\times d}$ and $A\in\BR^{d\times d}$ such that $G \succeq A$.
Based on the measure $\nu(\cdot,\cdot)$ and Lemma \ref{lemma:cur_iter_bound}, we provide the linear convergence of the distance to solution and the error of Hessian approximation as follows.
\begin{lemma}
For any $\rho$ satisfying $\rho\in(0,1 - d^{-1})$, there exist positive constants $r_0$ and $\sigma_0$ such that running LISR\Hyphdash*1 (Algorithm \ref{alg:iqn_algos}) with the initial conditions $\norm{x^0-x^*} \leq r_0$, $B_i^0 \succeq \omega^0 \nabla^2 f_i(z_i^0)$ and $\nu((\omega^0)^{-1}{B_i^0}, \nabla^2 f_i (x^0)) \leq \sigma_0$ for any~$i=1,\dots,n$ results in
 \begin{equation}
        \norm{x^{t+1} - x^*} \leq \rho^{\ceil{\frac{t+1}{n}}} \norm{x^0 - x^*}
        \label{eq:linear_recur_1}
    \end{equation}
and 
    \begin{equation}\label{eq:linear_recur_2}
    \small\begin{split}   
         \nu\big((\omega^{t+1})^{-1} B_{i_t}^{t+1}, \nabla^2 f_{i_t}(z_{i_t}^{t+1})\big) \leq \left(1 - \frac{1}{d}\right)^{\ceil{\frac{t+1}{n}}} \delta,
    \end{split}
    \end{equation}
    where  $M = \Tilde{L} / \mu^{{3}/{2}}$, $\omega_t$ follows the definition in (\ref{greedy_sr1_omega}) and
    \begin{align*}
     \delta \coloneqq \left(\sigma_0 +\frac{4M d L^{{3}/{2}} \mu^{-1} r_0 }{1 - (1-d^{-1})^{-1} \rho}\right)\exp\left(\frac{4M \sqrt{L} r_0}{1 - \rho}\right)   
    \end{align*}
    \label{lemma:core_recur}
\end{lemma}
\begin{remark}
    In the proof of Lemma \ref{lemma:core_recur}, we show that the relation $(\omega^{t+1})^{-1} B_{i_t}^{t+1} \succeq \nabla^2 f_{i_t}(z_{i_t}^{t+1})$ holds for each iteration. This guarantees the update rule (\ref{cur_iter_update}) and the error measure  $\nu\big((\omega^{t+1})^{-1} B_{i_t}^{t+1}, \nabla^2 f_{i_t}(z_{i_t}^{t+1})\big)$ are well-defined.
\end{remark}

We establish the mean-superlinear convergence based on Lemma~\ref{lemma:core_recur}. Specifically, we have the following result.
\begin{lemma}
    Following the initial conditions of Lemma~\ref{lemma:core_recur},
    the sequence of iterates generated by the LISR\Hyphdash*1 method (Algorithm~\ref{alg:iqn_algos}) satisfies
    \begin{equation*}
        \norm{x^{t+1} - x^*} \leq \left(1 - \frac{1}{d}\right)^{\ceil{\frac{t+1}{n}}}\cdot\frac{1}{n}\sum_{i=1}^n \norm{x^{t+1-i} - x^*}.
    \end{equation*}
        \label{lemma:sr1_avg}
\end{lemma}


Using Lemma \ref{lemma:sr1_avg}, we can achieve the local superlinear convergence rate of the proposed LISR\Hyphdash*1 method by induction.
We formally present our main result as follows.
\begin{theorem}
    For the sequence $\{x_t\}$ generated by LISR\Hyphdash*1  (Algorithm~\ref{alg:iqn_algos}) with the initial conditions shown in  Lemma~\ref{lemma:core_recur}, there exists a sequence $\{\zeta^l\}$ such that $\norm{x^{t+1} - x^*} \leq \zeta^{\floor{{t}/{n}}}$ for any $t\geq 1$ and it satisfies
    \begin{equation}
        \zeta^l \leq r_0 \left(1 - \frac{1}{d}\right)^{\frac{(l+2)(l+1)}{2}}.
        \label{main_result_eq}
    \end{equation}
    \label{thm:glins_res}
\end{theorem}

\subsection{Discussion}

The convergence analysis in the last subsection shows that LISR\Hyphdash*1 enjoys the condition number-free superlinear convergence rate, which is significantly better than all of the existing incremental fashion quasi-Newton methods (see Table \ref{tab:theoretic_res}).
The improvement is due to that we adopt the greedy SR1 update to maintain the Hessian estimator in formula (\ref{greedy_sr1_update-0}) and the analysis characterizes the Hessian approximation error by the measure $\nu(\cdot,\cdot)$ defined in (\ref{nudef}).
In contrast, the prior methods IGS~\cite{gao2020incremental} and SLIQN~\cite{lahoti2023sharpened} only consider the general Broyden family update and characterize the Hessian approximation error by the measure
$\sigma(G, A) = {\rm tr}(A^{-1}(G - A))$ for positive definite $G\in\BR^{d\times d}$ and $A\in\BR^{d\times d}$, which leads to additional dependency on condition number in the superlinear convergence rate.\footnote{These work present their theoretical results by analyzing BFGS update, while their analysis can be directly applied to the general Broyden family update and achieves the identical convergence rate.} 
On the other hand, the implementations of these methods are more complicated than ours. 
Concretely, IGS requires scaling a Hessian estimator at each iteration which results in $\fO(d^3)$ computational cost, and
SLIQN maintains $B_{i_t}^{t+1}$ by a combination of secant equation-based and greedy Broyden family updates while our LISR-1 only has one step of greedy SR1 update~(\ref{greedy_sr1_update-0}).

\begin{figure*}[!tb]
\centering
\begin{tabular}{ccc}
     \includegraphics[scale=0.45]{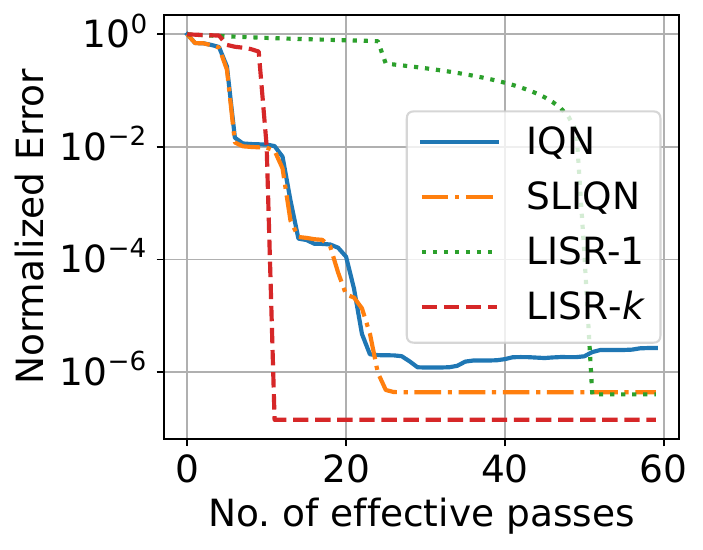}
     & \includegraphics[scale=0.45]{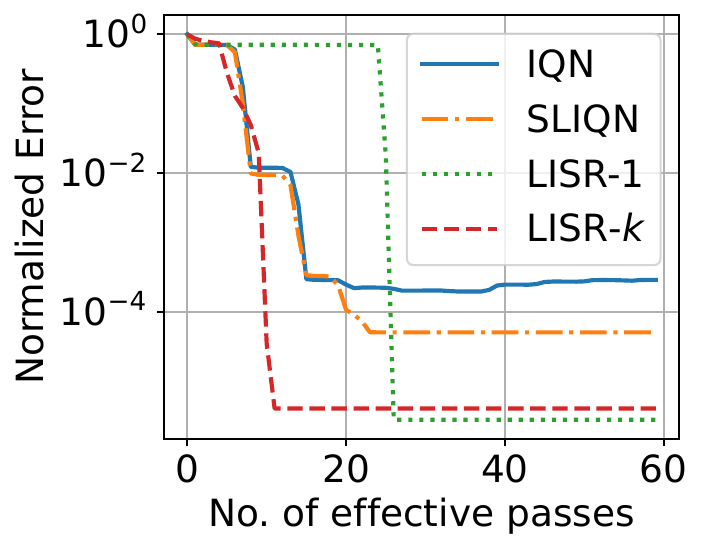}
     & \includegraphics[scale=0.45]{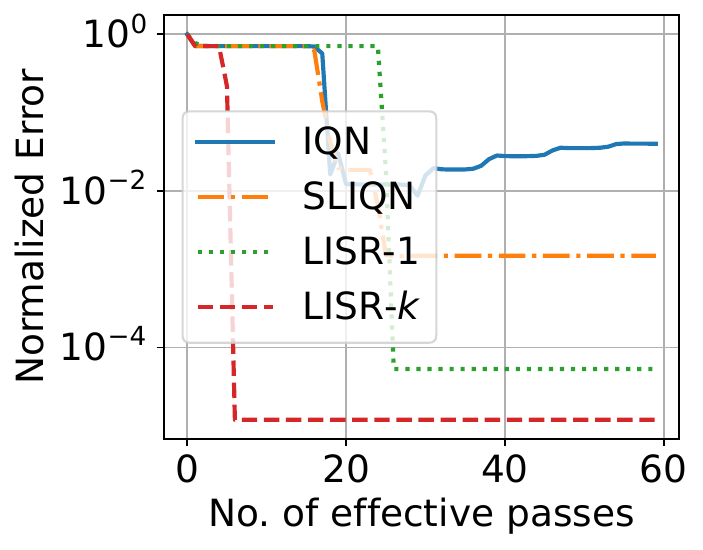} \\
     (a) $\xi = 4$, $\varkappa=3.03\times10^2$ & (b) $\xi = 8$, $\varkappa=3.12\times10^4$ & (c) $\xi = 12$, $\varkappa=3.12\times10^6$
\end{tabular}
\caption{Normalized error vs. the number of effective passes for the quadratic programming problem.}
\label{fig:quadratic_result}
\end{figure*}

\section{Extension to Block Quasi-Newton Methods}\label{sec:block_update}

It is possible to incorporate the idea of block quasi-Newton methods into the framework of the LISR-1.
Specifically, we only need to modify Line 6 of Algorithm \ref{alg:iqn_algos} by replacing the update rule (\ref{greedy_sr1_update-0}) with 
\begin{equation}\label{greedy_srk_update}
\small
\begin{split}    
\!\!\!\!B_{i_t}^{t+1}\!=\!\omega^{t+1} \textrm{SR-}k(B_{i_t}^{t},\!\nabla^2 f_{i_t}(z_{i_t}^{t+1}),\!\Bar{U}(B_{i_t}^{t},\!\nabla^2 f_{i_t}(z_{i_t}^{t+1}))),
\end{split}
\end{equation}
where $\Bar U(\cdot,\cdot)$ contains greedy directions which is defined as
\begin{equation}
    \Bar{U}(G, A)
    = E_k(G-A).
    \label{greedy_matrix}
\end{equation}
We name the variant of LISR-1 with the above modification as Lazy Incremental Symmetric Rank-$k$ (LISR-$k$) method.

The LISR-$k$ method requires $\fO(k d^2)$ flops in each iteration.
Since we typically set $k$ to be much smaller than $d$, such computational cost is acceptable.
Similar to the previous analysis, the cost of LISR-$k$ is dominated by maintaining the inverse of the sum of individual
Hessian estimators 
\begin{equation*}
    \Bar{B}^{t+1} \coloneqq \sum_{i=1}^n B_i^{t+1},
\end{equation*}
which can be written as $\Bar{B}^{t} + B_{i_t}^{t+1} - B_{i_t}^{t}$.
The main difference between the two algorithms is the update on $\Bar B^{t+1}$ (its inverse) in the case of $t~{\rm mod}~n \neq 0$. For the LISR-$k$ method, we have
\begin{equation*}
\begin{split}
        \Bar{B}^{t+1} = &\Bar{B}^{t} - V^{t} \left((\Bar{U}^{t})^{\top} V^{t}\right)^{-1} (V^{t})^{\top},
\end{split}
\end{equation*}
where we define $V^{t}=\left(B_{i_t}^{t}-\nabla^2 f_{i_{t}}(z_{i_{t}}^{t+1})\right) \Bar{U}^{t}\in\BR^{d\times k}$ and $\Bar{U}^{t}=\Bar{U}(B_{i_{t}}^{t}, \nabla^2 f_{i_{t}}(z_{i_{t}}^{t+1}))\in\BR^{d\times k}$. Applying the Sherman-Morrison formula, we achieve
\begin{equation}\label{eq:inverse-k}\small
\begin{split}
    (\Bar{B}^{t+1})^{-1} = & (\Bar{B}^{t})^{-1} + (\Bar{B}^{t})^{-1}  V^{t} (D^{t})^{-1} (V^{t})^{\top} (\Bar{B}^{t})^{-1},
\end{split}
\end{equation}
where $D^{t} = (\Bar{U}^{t})^{\top} V^{t} -  (V^{t})^{\top} (\Bar{B}^{t})^{-1} V^{t}\in\BR^{k\times k}$. 
It can be observed that constructing $D^t$ takes $\fO(k d^2)$ flops for given $(\Bar B_t)^{-1}$. 
Additionally, the complexity of computing $(D^t)^{-1}$ is not the leading cost since we take $k\ll d$.
Hence, the total cost for computing Eq.~(\ref{eq:inverse-k}) is $\fO(k d^2)$ flops.
Similar to LISR-1, the setting of $\omega^{t+1}$ guarantees the scaling occurs once every $n$ iterations and its amortized per-iteration complexity is no more than $\fO(k d^2)$ flops.


\paragraph{Even Faster Convergence Rate} 

The rank-$k$ update in the LISR-$k$ leads to sharper upper bounds on the distance to optimal solution and approximation error of Hessian estimators.
Compared with Lemma \ref{lemma:sr1_avg}, the LISR-$k$ holds the following tighter upper bounds
\begin{equation*}
    \norm{x^{t+1} - x^*} \leq \left(1 - \frac{k}{d}\right)^{\ceil{\frac{t+1}{n}}}\frac{1}{n}\sum_{i=1}^n \norm{x^{t+1-i} - x^*}
\end{equation*}
and
\begin{equation*}
    \nu((\omega^{t+1})^{-1} B_{i_t}^{t+1}, \nabla^2 f_{i_t}(z_{i_t}^{t+1})) \leq \left(1 - \frac{k}{d} \right)^{\ceil{\frac{t+1}{n}}} \delta.
\end{equation*}

Consequently, we can show the mean-superlinear convergence result like Lemma \ref{lemma:sr1_avg}, and the new result improves the base of convergence rate from $1-d^{-1}$ to $1-kd^{-1}$. 
Finally, we achieve the main result of the LISR-$k$ method as follows.
\begin{theorem}
    We follow the initial conditions of Lemma~\ref{lemma:core_recur} but initialize $\rho$ with $\rho\in(0,1-kd^{-1})$. For the sequence of iterates $\{x^t\}$ generated by the LISR-$k$~method, there exists sequence $\{\zeta^l\}$ such that $\norm{x^t-x^*} \leq \zeta^{\floor{(t-1)/n}}$ for any $t\geq 1$ and it satisfies
    \begin{equation}
        \zeta^l \leq r_0 \left(1 - \frac{k}{d}\right)^{\frac{(l+2)(l+1)}{2}}.
        \label{srk_result_eq}
    \end{equation}
    \label{thm:glins+_res}
\end{theorem}
\begin{remark}
For $k = 1$, the LISR-$k$ method degenerates to the LISR-1 method.
For $k\geq 2$, the superlinear convergence rate of LISR-$k$ (Theorem \ref{thm:glins+_res}) is strictly tighter than the counterpart of LISR-$1$ (Theorem \ref{thm:glins_res}).
\end{remark}

\section{Experiments} \label{sec:experiment}
We compare the proposed methods LISR-1 and LISR-$k$ with baseline methods including IQN~\cite{mokhtari2018iqn} and SLIQN~\cite{lahoti2023sharpened}.
We test all methods on the problems of quadratic programming and regularized logistic regression.
For the LISR-$k$ method, we set $k=5$ for all of the cases.
For the fairness of comparison, we run all algorithms from the same initial point.

\subsection{Quadratic Function Minimization}
We consider the following quadratic function minimization problem 
\begin{equation}
\small
    \min_{x\in \BR^d} f(x) \coloneqq \frac{1}{n} \sum_{i=1}^n \Big( \frac{1}{2} \langle x, A_i x\rangle + \langle b_i, x\rangle\Big),
\end{equation}
where $A_i \in \BR^{d \times d}$ is positive definite and $b_i\in\BR^d$. 
Following the setup of~\citet{mokhtari2018iqn}, we let each $A_i$ be diagonal matrix by setting the first half of diagonal entries be independent uniformly sampled from $[1, 10^{\xi / 2}]$ while the others are independent uniformly sampled from $[10^{-\xi / 2}, 1]$, where $\xi>0$ is the parameter that affects the condition number of the problem.
For each $b_i$, we let its entries be independently uniformly sampled from $[0, 10^3]$.

We run the experiments by taking $n=1000$, $d=50$ and $\xi\in\{4,8,12\}$, and we present the results in Figure~\ref{fig:quadratic_result}. 
We observe that the condition number heavily affects the convergence behaviors of IQN and SLIQN, while the proposed methods LISR-1 and LISR-$k$ are insensitive to the varying condition numbers. 
These results validate our theoretical analysis since we have shown the superlinear convergence rates of our methods do not depend on the condition number.

\begin{figure*}[!tb]
\centering
\begin{tabular}{ccc}
     \includegraphics[scale=0.45]{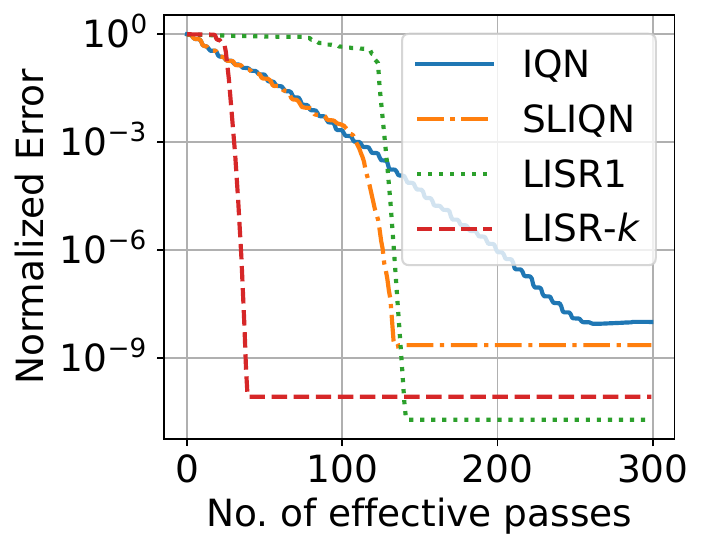}
     & \includegraphics[scale=0.45]{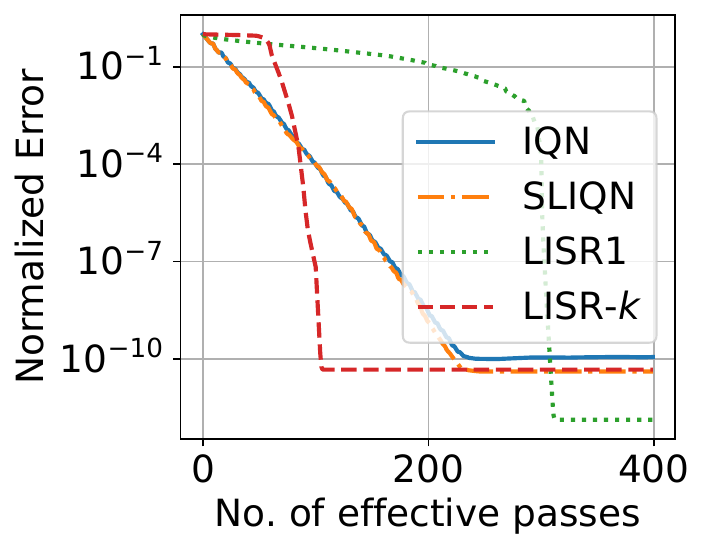}
     & \includegraphics[scale=0.45]{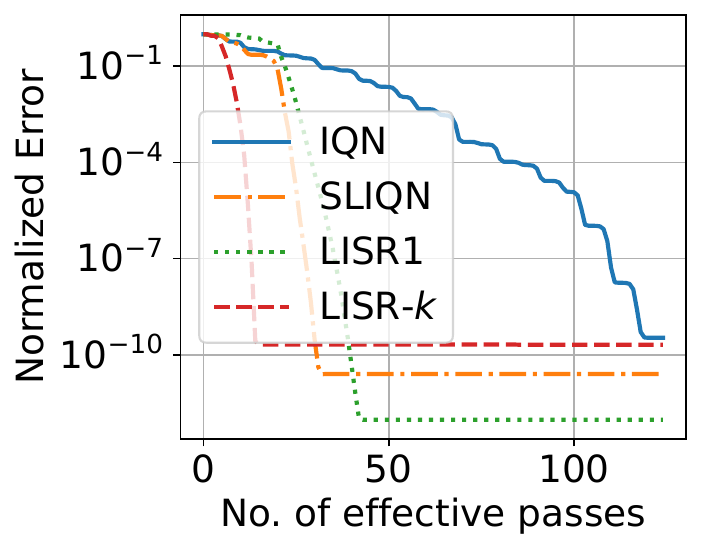} \\
     (a) a9a & (b) w8a &  (c) ijcnn1 \\[0.1cm]
     \includegraphics[scale=0.45]{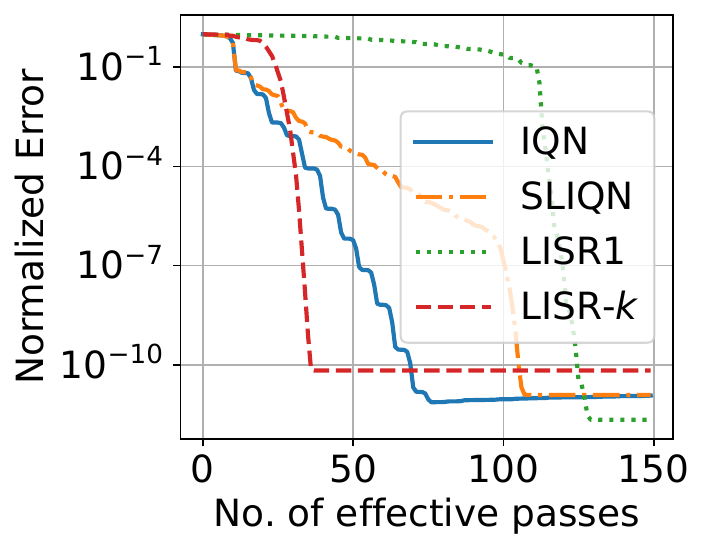}
     & \includegraphics[scale=0.45]{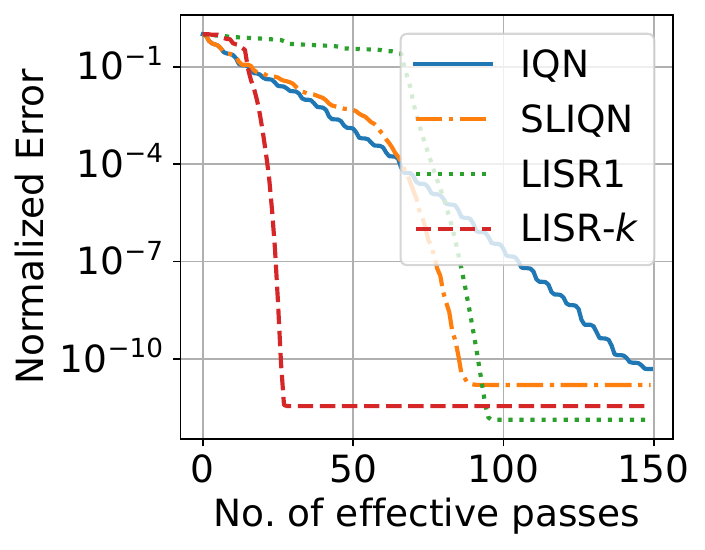}
     & \includegraphics[scale=0.45]{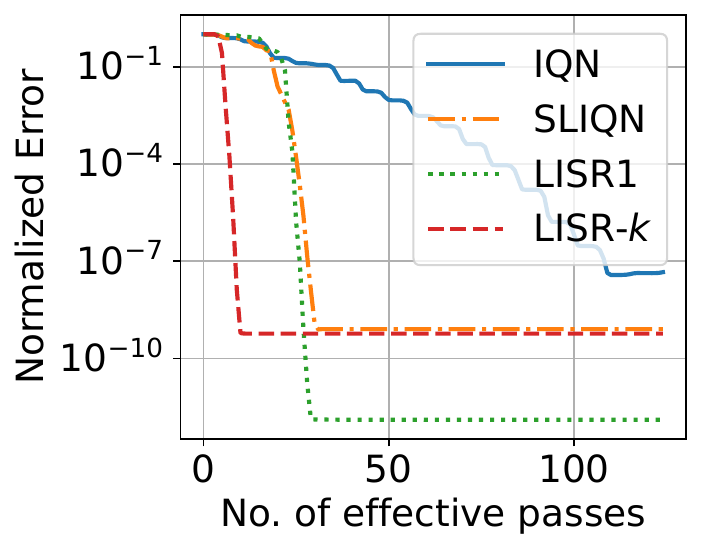} \\   
     (e) mushrooms & (f) phishing &  (g) svmguide3 \\[0.1cm]
     \includegraphics[scale=0.45]{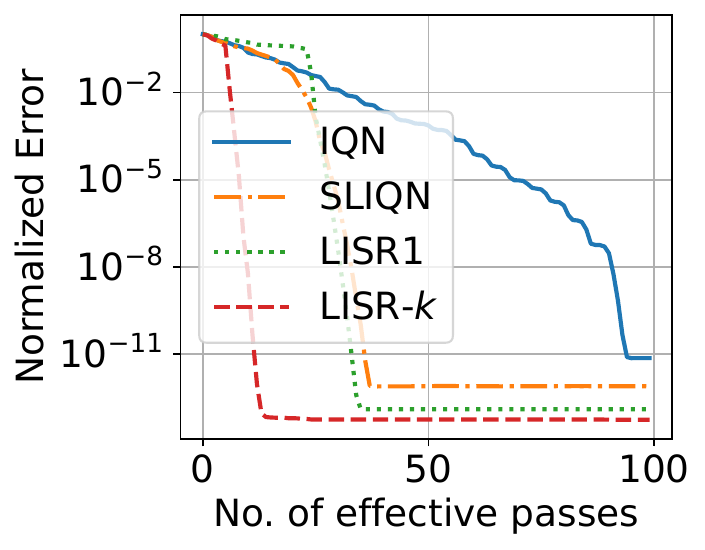}
     & \includegraphics[scale=0.45]{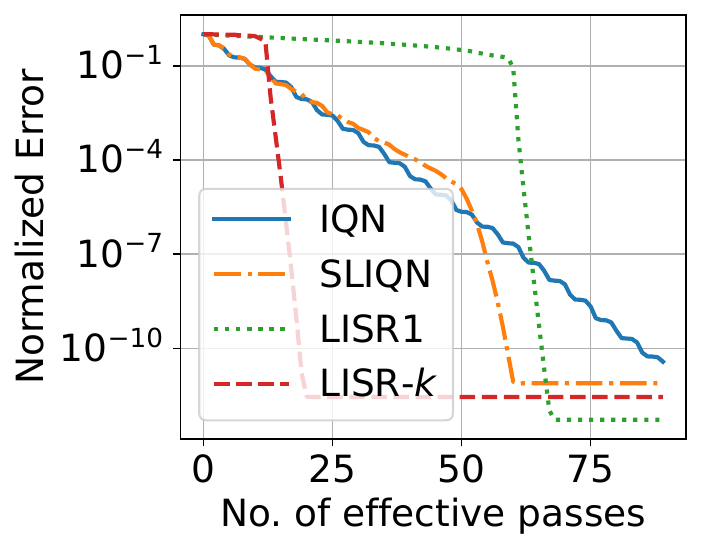}
     & \includegraphics[scale=0.45]{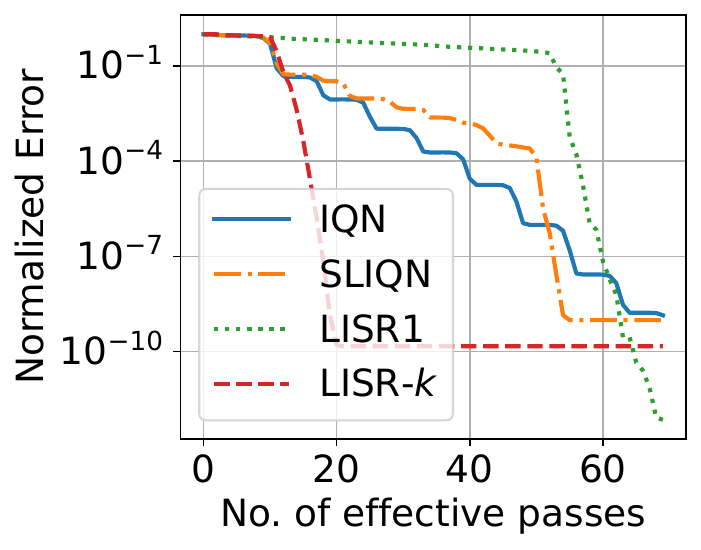}    \\
     (h) german.numer & (i) splice &  (j) covtype \\[0.1cm]
\end{tabular}
\caption{Normalized error vs. the number of effective passes for the regularized logistic regression problem on several real-world datasets . }
\label{fig:general_res}
\end{figure*}

\subsection{Regularized Logistic Regression}

We consider $\ell_2$-regularized logistic regression problem 
\begin{equation}
\begin{split}
     \min_{x\in\BR^d} f(x)\!\coloneqq\!\frac{1}{n} \sum_{i=1}^n\ln(1\!+\!\exp(- y_i \langle x, z_i \rangle))\!+\!\frac{\lambda}{2}\!\norm{x}^2\!,
\end{split}
\label{logistic_obj}
\end{equation}
where $z_i\in\BR^d$ is the feature of the $i$-th training sample and $y_i \in \{1,-1\}$ is the corresponding labels.
We conduct our experiments on nine real-world datasets (``a9a'', ``w8a'', ``ijcnn'', ``mushrooms'', ``phishing'', ``svmguide3'', ``german.numer'', ``splice'' and `covtype'') from LIBSVM repository.  
We take $\lambda=10^{-3}$ for ``a9a'', ``mushrooms'', ``svmguide3'', ``german.numer'', ``covtype'' and $\lambda=10^{-4}$ for others.

We present the experimental results in Figure \ref{fig:general_res}.
We observe that the proposed LISR-$k$ significantly outperforms other methods on all datasets. 
The LISR-1 enjoys a faster convergence rate than IQN and SLIQN when it starts to converge, while it may be slower at the early stage. 
We conjecture that IQN and SLIQN contain the steps of classical quasi-Newton updates.
By accessing the exact gradient information, \citet{rodomanov2021new,rodomanov2021rates} theoretically showed that classical quasi-Newton methods converge faster than greedy quasi-Newton methods at the early stage.
We empirically observe similar results for incremental quasi-Newton methods, while the rigorous theory for such a phenomenon is still unclear.
On the other hand, the block update in LISR-$k$ leads to much better Hessian estimators. Hence, the early stage of LISR-$k$ only contains a few iterations.

\section{Conclusion}\label{sec:conclusion}
This paper has proposed the efficient incremental quasi-Newton method called LISR-1 and its extension named LISR-$k$ method for the finite-sum convex optimization. 
We have theoretically shown the proposed methods enjoy faster superlinear convergence rates than the state-of-the-art incremental quasi-Newton methods.
The numerical experiments on quadratic programming and regularized logistic regression also validate the advantages of the proposed methods over existing IQN baselines.

In future work, it is interesting to study incremental quasi-Newton methods for more general settings, such as minimizing nonconvex functions~\cite{wang2017stochastic,yang2021stochastic}.
It is also possible to leverage the idea to design efficient incremental quasi-Newton methods for solving minimax problems~\cite{liu2022quasi,liu2022partial} or nonlinear equations~\cite{liu2023block}.

\section*{Acknowledgments}
This research/project is supported by the National Research Foundation, Singapore under its AI Singapore Programme (AISG Award No: AISG2-PhD-2023-08-043T-J).
This research is part of the programme DesCartes and is supported by the National Research Foundation, Prime Minister’s Office, Singapore under its Campus for Research Excellence and Technological Enterprise (CREATE) programme.
Luo Luo is supported by National Natural Science Foundation of China (No. 62206058) and Shanghai Sailing Program (22YF1402900).

\bibliography{aaai24}
\appendix\onecolumn

\newpage

The appendix is organized as below,
Section \ref{appendix:key_lemmas} introduces several key lemmas which are essential for the convergence analysis of the LISR-1 and LISR-$k$.
Section \ref{appendix:main_result_proof} presents the proof of lemmas and theorems introduced in Section \ref{sec:methodology} and Section \ref{sec:block_update}.
Section \ref{appendix:efficient_impl} shows the efficient implementation of the LISR-1 and LISR-$k$ method which achieve a computation complexity of $\fO(d^2)$ at every iteration. 
Section \ref{appendix:more_exps} presents more details of the dataset. 
It also presents additional comparison results on the quadratic function minimization task and general function minimization task.

\section{Established Results} \label{appendix:key_lemmas}
In this section, we revisit some key lemmas which are essential for the analysis of the proposed algorithms. 
\begin{lemma}[\citet{rodomanov2021greedy}]
    Let $f$ be a strongly self-concordant function with some constant $M$, and let $r \coloneqq   \norm{y - x}_{\nabla^2 f(x)} $ for any $x, y \in \fR^d$. Then,
    \begin{equation}
        \frac{\nabla^2 f(x)}{1 + Mr} \preceq \nabla^2 f(y) \preceq (1 + M r) \nabla^2 f(x).
        \label{matrix_approx_1}
    \end{equation}
    Also, for $J \coloneqq \int_{0}^1 \nabla^2 f(x + t(y - x))dt$, we have
    \begin{equation}
        \frac{\nabla^2 f(x)}{1 + \frac{M r}{2}} \preceq J \preceq \left(1 + \frac{M r}{2}\right) \nabla^2 f(x),
    \end{equation}
    \begin{equation}
        \frac{\nabla^2 f(y)}{1 + \frac{Mr}{2}} \preceq J \preceq \left(1 + \frac{M r}{2}\right)\nabla^2 f(y).
    \end{equation}
    \label{lemma:matrix_approx}
\end{lemma}

\begin{lemma}[Banach's Lemma]
    Let $A \in \BR^{d \times d}$ be a matrix such that its norm satisfies $\norm{A} < 1$. Then the matrix $(I + A)$ is invertible and 
    \begin{equation*}
        \frac{1}{1 + \norm{A}} < \norm{(I + A)^{-1}} < \frac{1}{1 - \norm{A}}. 
    \end{equation*}
    \label{lemma:banach}
\end{lemma}
\begin{lemma}[\citet{ye2021explicit}]
    Consider positive definite matrices $A, G \in \BR^{d \times d}$ and suppose $G_{+} \coloneqq SR1(G, A, u)$, where $u \neq 0$. Suppose for some $\eta > 1$, we have
    \begin{equation*}
        A \preceq G \preceq \eta A,
    \end{equation*}
    then, for any $u \neq 0$, we also have:
    \begin{equation*}
        A \preceq G_{+} \preceq G \preceq \eta A.
    \end{equation*}
    \label{lemma:sr1_one_step}
\end{lemma}
\begin{lemma}[\citet{lin2022explicit}]
    Consider positive definite matrices $A, G \in \BR^{d \times d}$ such that $A \preceq G$. Suppose that $G_{+} \coloneqq \textrm{SR1}(G, A, \Tilde{u}(G, A))$, where $\Tilde{u}(G, A)$  (\eqref{sr_greedy_vec}) is the greedy vector of $G$ with respect to $A$. Then, the following holds:
    \begin{equation*}
        \tau(G_{+}, A) \leq \left(1 - \frac{1}{d}\right) \tau(G, A).
    \end{equation*}
    \label{lemma:sr1_one_step_2}
\end{lemma}
\begin{lemma}[\citet{liu2023symmetric}]
    For any positive-definite matrices $A \in \BR^{d \times d}$ and $G \in \BR^{d \times d}$ with $A \preceq G \preceq \eta A$ for some $\eta \geq 1$, we let $G_{+} = \mathrm{SR}\textrm{-}k(G, A, U)$ for some full rank matrix $U \in \BR^{d\times k}$.
    Then it holds that
    \begin{equation}
        A \preceq G_{+} \preceq \eta A.
    \end{equation}
    \label{lemma:srk_update_consist}
\end{lemma}
\begin{lemma}[\citet{liu2023symmetric}]
    Suppose that 
    \begin{equation*}
        G_+ = \mathrm{SR}\textrm{-}k(G, A, U),
    \end{equation*}
    where $G \succeq A \in \BR^{d \times d}$ and select $U \in \BR^{d \times k}$ by \eqref{greedy_matrix}. Then we have
    \begin{equation}
        \tau(G_{+}, A) \leq \left( 1 - \frac{k}{d}\right) \tau(G, A).
    \end{equation}
    \label{lemma:srk_one_step}
\end{lemma}

\section{Convergence Analysis of the LISR-1/LISR-$k$ method} \label{appendix:main_result_proof}
In this section, we present the theoretical analysis of the Theorem \ref{thm:glins_res} and Theorem \ref{thm:glins+_res}.
To facilitate the analysis of main theorems, we introduce the following quantity given two PSD matrices $G$ and $A$ such that $G \succeq A$:
\begin{equation*}
    \tau(G, A) = \rm tr(G - A).
\end{equation*}
\subsection{Supporting Lemmas}
The norm of the difference between two PSD matrices can be upper bounded as below:
\begin{lemma}
    For all positive definite matrices $A, G \in \BR^{d \times d}$, if $A \preceq G$ and $A \preceq L I$, then
    \begin{equation*}
        \norm{G - A} \leq \frac{d L \tau(G, A)}{\mathrm{tr}(A)} = \mu \nu(G, A).
    \end{equation*}
    \label{lemma:norm_bound}
\end{lemma}
\begin{proof}
    For any PSD matrix $X$, let $\lambda_{max}(X)$ denote the $i$-th largest eigenvalue of $X$, we have
    \begin{align*}
        \norm{G - A} = \lambda_1(G- A) \leq \sum_{i=1}^d \lambda_i(G - A) = tr(G - A) = \tau(G, A).
    \end{align*}
    Furthermore, we can bound $\mathrm{tr}(A)$ by:
    \begin{equation*}
        \mathrm{tr}(A) \leq d L.
    \end{equation*}
    Therefore, we have:
    \begin{equation*}
        \frac{d L \tau(G, A)}{\mathrm{tr}(A)} \geq \frac{d L \norm{G - A}}{d L} = \norm{G - A}.
    \end{equation*}
\end{proof}
Now we show how the Hessian approximation metric $\nu(B, \nabla^2 f(x))$ changes after applying one step of greedy SR1 update:
\begin{lemma}
    Let $f: \BR^d \to \BR$ be a real-valued function that is $\mu$-strongly convex, $L$-smooth, and $M$-strongly self-concordant. Let $x, x_{+} \in \BR^d \setminus \{0\}$ and $B$ be a matrix such that $B \geq \nabla^2 f(x)$. Define the constant $r \coloneqq \norm{x_{+} - x}_{\nabla^2 f(x)}$ and the matrix $P \coloneqq \left(1 + \frac{M r}{2}\right)^2 B$. Consider the following SR-$k$ updates:
    \begin{equation*}
        B_{+} \coloneqq \mathrm{SR}\textrm{-}k(P, \nabla^2 f(x_{+}), \Bar{U}(P, \nabla^2 f(x_{+}))).
    \end{equation*}
    Here the vector $\Bar{U}(P, \nabla^2 f(x_{+}))$ is the greedy matrx defined in \eqref{greedy_matrix}. Then $B_{+} \succeq \nabla^2 f(x_{+})$ and
    \begin{equation*}
        \frac{\tau(B_{+}, \nabla^2 f(x_{+}))}{\mathrm{tr}(\nabla^2 f(x_{+}))} \leq \left(1 - \frac{k}{d}\right)\left(1 + \frac{Mr}{2}\right)^4 \left(\frac{\tau (B, \nabla^2 f(x))}{\mathrm{tr}(\nabla^2 f(x))} + 2 M r  \right).
    \end{equation*}
    \label{lemma:one_step_approx_srk}
\end{lemma}
\begin{proof}
    Since $B \succeq \nabla^2 f(x)$, we have:
    \begin{align*}
        P =& (1 + \frac{M r}{2})^2 B \succeq (1 + \frac{M r}{2})^2 \nabla^2 f(x) \succeq (1 + \frac{Mr}{2}) K \succeq \nabla^2 f(x_+),
    \end{align*}
    where the matrix $K \coloneqq \int_{\tau = 0}^1 f(x + \tau (x_{+} - x)) d \tau$. 
    The first inequality follows from the problem assumption and the last two inequalities are due to Lemma \ref{lemma:matrix_approx}. 
   
    
    Apply Lemma \ref{lemma:srk_one_step}, and we have:
    \begin{align*}
        \tau(B_{+}, \nabla^2 f(x_{+})) \preceq & \left(1 - \frac{k}{d} \right) \tau(P, \nabla^2 f(x_{+})) .
    \end{align*}
    
    The last term can be bounded as 
    \begin{align*}
         & \mathrm{tr}(P - \nabla^2 f(x_{+})) \\
         \overset{(\ref{matrix_approx_1})}{\leq} & \mathrm{tr}\left(\left(1 + \frac{Mr}{2}\right)^2 B - \frac{\nabla^2 f(x)}{1 + M r}\right) \\
         = & \left(1 + \frac{Mr}{2}\right)^2 \tau (B, \nabla^2 f(x)) + \left(\left(1 + \frac{Mr}{2}\right)^2 - \frac{1}{1 + M r}\right) \mathrm{tr}(\nabla^2 f(x)) \\
         \leq & \left(1 + \frac{Mr}{2}\right)^2 \tau (B, \nabla^2 f(x)) + \left(\left(1 + \frac{Mr}{2}\right)^2 + Mr - 1\right) \mathrm{tr}(\nabla^2 f(x)) \\
         = & \left(1 + \frac{Mr}{2}\right)^2 \tau (B, \nabla^2 f(x)) +  M r \left( 2 + \frac{Mr}{4} \right) \mathrm{tr}(\nabla^2 f(x)) \\
         \overset{(\ref{matrix_approx_1})}{\leq} & \left(1 + \frac{Mr}{2}\right)^2 (1 + M r)\left(\frac{\tau (B, \nabla^2 f(x))}{\mathrm{tr}(\nabla^2 f(x))} +  2 M r  \right) \mathrm{tr}(\nabla^2 f(x_{+})), \\
    \end{align*}
    By $ 1 + Mr \leq (1 + \frac{Mr}{2})^2$, and rearrange the terms we have:
    \begin{equation*}
        \frac{\tau(B_{+}, \nabla^2 f(x_{+}))}{\mathrm{tr}(\nabla^2 f(x_{+}))} \leq c\left(1 + \frac{Mr}{2}\right)^4 \left(\frac{\tau (B, \nabla^2 f(x))}{\mathrm{tr}(\nabla^2 f(x))} + 2 M r  \right),
    \end{equation*}
    where $c\coloneqq 1 - k d^{-1}$.
\end{proof}
Applying the result of Lemma \ref{lemma:one_step_approx_srk}, we can upper bound the Hessian approximation error with the following lemma:
\begin{lemma}
    Let $f: \BR^d \to \BR$ be a real-valued function that is $\mu$-strongly convex, $L$-smooth, and $M$-strongly self-concordant. Let $\Tilde{x} \in \BR^d$ be some fixed vector and $0 \leq \gamma < 1$ be some fixed constant such that the sequence $\{ x^t \}$, for all $t \in [T]$ satisfies
    \begin{equation}
        \norm{x^t - \Tilde{x}} \leq \gamma^t \norm{x^0 - \Tilde{x}}.
        \label{eq:dist_hypo_1}
    \end{equation}
    Define the constant $r_t \coloneqq \norm{x^t - x^{t-1}}_{\nabla^2 f(x^{t-1})}$ for every $t$. Let $B^0$ be a matrix such that it satisfies $B_0 \succeq \nabla^2 f(x^0)$. Consider the following SR-$k$ update:
    \begin{align*}  
        B^t \coloneqq & \mathrm{SR}\textrm{-}k(P^{t-1}, \nabla^2 f(x^t), \Tilde{U}(P^{t-1}, \nabla^2 f(x^t))),
    \end{align*}
    where $P^{t-1} \coloneqq (1 + \frac{M r_t}{2})^2 B^{t-1}$, 
    and $\Tilde{U}(P^{t-1}, \nabla^2 f(x^t))$ is the greedy matrix (\ref{greedy_matrix}). Then the following holds for all $t \in [T]$:
    \begin{equation}
        \nu(B^t, \nabla^2 f(x^t)) \leq  \left(1 - \frac{k}{d}\right)^t e^ {\frac{4 M \sqrt{L}  \norm{x^0 - \Tilde{x}}}{1 - \gamma} } \left( \nu (B^{0}, \nabla^2 f(x^{0})) + \norm{x^0 - \Tilde{x}} \frac{4 d M L^{\frac{3}{2}} \mu^{-1} }{1 - (1 - k / d)^{-1} \gamma}  \right) .
    \end{equation}
    \label{lemma:srk_recur}
\end{lemma}
\begin{proof}
    From Lemma \ref{lemma:one_step_approx_srk}, it can be shown that $B^t \succeq \nabla^2 f(x^t)$ for $t \in [T]$. Therefore, $\tau(B^t, \nabla^2 f(x^t))$ and $\nu(B^t, \nabla^2 f(x^t))$ are both well defined.

    First, we can apply the triangle inequality on $r_t$ to obtain:
    \begin{align*}
        r_t = \norm{x^t - x^{t-1}}_{\nabla^2 f(x^{t-1})} \leq & \sqrt{L} \norm{x^t - x^{t-1}} \\
        \leq & \sqrt{L}(\norm{x^t - \Tilde{x}} + \norm{x^{t-1} - \Tilde{x}}) \leq 2 \sqrt{L} \gamma^{t-1} \norm{x^0 - \Tilde{x}},
    \end{align*}
    Where the first inequality is due to the $L$-smoothness of function $f(\cdot)$. The last inequality follows from \eqref{eq:dist_hypo_1}. 
    Applying Lemma \ref{lemma:one_step_approx_srk} for the SR-$k$ update, we have:
    \begin{align*}
         \frac{\tau(B^{t}, \nabla^2 f(x^{t}))}{\mathrm{tr}(\nabla^2 f(x^{t}))} \leq & c\left(1 + \frac{M r_{t-1}}{2}\right)^4 \left(\frac{\tau (B^{t-1}, \nabla^2 f(x^{t-1}))}{\mathrm{tr}(\nabla^2 f(x^{t-1}))} + 2 M r_{t-1}  \right) \\
         \leq & c e^ {2 M r_{t-1}} \left(\frac{\tau (B^{t-1}, \nabla^2 f(x^{t-1}))}{\mathrm{tr}(\nabla^2 f(x^{t-1}))} + 2 M r_{t-1}  \right)\\
         \leq & c e^ {4 M \sqrt{L} \gamma^{t-1} \norm{x^0 - \Tilde{x}}} \left(\frac{\tau (B^{t-1}, \nabla^2 f(x^{t-1}))}{\mathrm{tr}(\nabla^2 f(x^{t-1}))} + 4 M \sqrt{L} \gamma^{t-1} \norm{x^0 - \Tilde{x}}  \right),
    \end{align*}
    where $c = 1 - k d^{-1}$ for the SR-$k$ update.
    The second inequality follows from $1 + x \leq e^x$. Expand the recursion, then:
    \begin{align*}
         \frac{\tau(B^{t}, \nabla^2 f(x^{t}))}{\mathrm{tr}(\nabla^2 f(x^{t}))} \leq & c e^ {4 M \sqrt{L} \gamma^{t-1} \norm{x^0 - \Tilde{x}}} \left(\frac{\tau (B^{t-1}, \nabla^2 f(x^{t-1}))}{\mathrm{tr}(\nabla^2 f(x^{t-1}))} + 4 M \sqrt{L} \gamma^{t-1} \norm{x^0 - \Tilde{x}}  \right) \\
         \leq & c^2 e^ {4 M \sqrt{L}  \norm{x^0 - \Tilde{x}} (\gamma^{t-1} + \gamma^{t-2})} \frac{\tau (B^{t-2}, \nabla^2 f(x^{t-2}))}{\mathrm{tr}(\nabla^2 f(x^{t-2}))} + \\
         & 4 M \sqrt{L}  \norm{x^0 - \Tilde{x}}\left( c \gamma^{t-1} e^ {4 M \sqrt{L} \gamma^{t-1} \norm{x^0 - \Tilde{x}}} + c^2 \gamma^{t-2} e^ {4 M \sqrt{L} (\gamma^{t-2} + \gamma^{t-1}) \norm{x^0 - \Tilde{x}}} \right)  \\
         \leq & c^t e^ {4 M \sqrt{L}  \norm{x^0 - \Tilde{x}} \sum_{j =0}^{t-1}\gamma^{j}} \frac{\tau (B^{0}, \nabla^2 f(x^{0}))}{\mathrm{tr}(\nabla^2 f(x^{0}))} + \\
         & 4 M \sqrt{L}  \norm{x^0 - \Tilde{x}}\left(\sum_{j=0}^{t-1} c^{t-j}\gamma^{j} e^ {4 M \sqrt{L}  \norm{x^0 - \Tilde{x}} \sum_{i=1}^{t-j} \gamma^{t-i}}\right)  \\
         \leq & c^t e^ {4 M \sqrt{L}  \norm{x^0 - \Tilde{x}} \sum_{j =0}^{+ \infty}\gamma^{j}} \frac{\tau (B^{0}, \nabla^2 f(x^{0}))}{\mathrm{tr}(\nabla^2 f(x^{0}))} + \\
         & 4 M \sqrt{L}  \norm{x^0 - \Tilde{x}} e^ {4 M \sqrt{L}  \norm{x^0 - \Tilde{x}} \sum_{i=1}^{+\infty} \gamma^{i}} \left(\sum_{j=0}^{t-1} c^{t-j}\gamma^{j} \right)  \\
         \leq & c^t e^ {\frac{4 M \sqrt{L}  \norm{x^0 - \Tilde{x}}}{1 - \gamma} } \left( \frac{\tau (B^{0}, \nabla^2 f(x^{0}))}{\mathrm{tr}(\nabla^2 f(x^{0}))} + \norm{x^0 - \Tilde{x}} \frac{4 M \sqrt{L}}{1 - c^{-1} \gamma}  \right) . 
    \end{align*}
    Multiply both sides of the above inequality by $d \kappa$ and we get the desired result. 
\end{proof}

\subsection{Proof of Lemma \ref{lemma:cur_iter_bound}}
For the completeness of the paper presentation, we show the proof of Lemma \ref{lemma:cur_iter_bound} below:
\begin{proof}
    For all $t \geq 0$, we define $H^t \coloneqq (\sum_{i=1}^n B_i^t)^{-1}$. From the update for $x^{t+1}$, we have:
    \begin{align*}
        \norm{x^{t+1} - x^*} =& H^t \left( \sum_{i=1}^n B_i^t z_i^t - \sum_{i=1}^n \nabla f_i(z_i^t) \right) - x^* \\
         = & H^{t} \left( \sum_{i=1}^n B_i^{t} (z_i^t - x^*) - \sum_{i=1}^n \nabla f_i(z_i^t) \right) \\
         = & H^{t} \left( \sum_{i=1}^n B_i^{t} (z_i^t - x^*) - \sum_{i=1}^n (\nabla f_i(z_i^t) - \nabla f_i(x^*))\right) \\
        = &  H^{t} \left( \sum_{i=1}^n B_i^{t} (z_i^t - x^*) - \sum_{i=1}^n \int_0^1\nabla^2 f(x^* + (z_i^t - x^*)v) (z_i^t - x^*) d v \right) \\
        = &  H^{t} \left( \sum_{i=1}^n (B_i^{t} - \nabla^2 f_i(z_i^t)) (z_i^t - x^*) + \sum_{i=1}^n (\nabla^2 f_i (z_i^t) - \int_0^1\nabla^2 f(x^* + (z_i^t - x^*)v) ) (z_i^t - x^*) d v \right),
    \end{align*}
    where the second last equality follows from the Fundamental Theorem of Calculus. Taking the norm on both sides and applying the triangle inequality, we have
    \begin{align*}
        \norm{x^{t+1} - x^*} \leq & \norm{H^{t}} \left( \sum_{i=1}^n \norm{B_i^{t} - \nabla^2 f_i(z_i^t)} \norm{z_i^t - x^*} + \sum_{i=1}^n \norm{\int_0^1 (\nabla^2 f_i (z_i^t) - \nabla^2 f(x^* + (z_i^t - x^*)v))  (z_i^t - x^*) d v } \right) \\
        \leq & \norm{H^{t}} \left( \sum_{i=1}^n \norm{B_i^{t} - \nabla^2 f_i(z_i^t)} \norm{z_i^t - x^*} + \sum_{i=1}^n \int_0^1 \norm{(\nabla^2 f_i (z_i^t) - \nabla^2 f(x^* + (z_i^t - x^*)v))  (z_i^t - x^*)} d v   \right) \\
        \leq & \norm{H^{t}} \left( \sum_{i=1}^n \norm{B_i^{t} - \nabla^2 f_i(z_i^t)} \norm{z_i^t - x^*} + \sum_{i=1}^n \int_0^1 \norm{\nabla^2 f_i (z_i^t) - \nabla^2 f(x^* + (z_i^t - x^*)v)}  \norm{z_i^t - x^*} d v   \right) \\
        \leq & \norm{H^{t}} \left( \sum_{i=1}^n \norm{B_i^{t} - \nabla^2 f_i(z_i^t)} \norm{z_i^t - x^*} + \sum_{i=1}^n \Tilde{L} \int_0^1 (1 - v) d v \norm{z_i^t - x^*}^2   \right) \\
        \leq & \norm{H^{t}} \left( \sum_{i=1}^n \norm{B_i^{t} - \nabla^2 f_i(z_i^t)} \norm{z_i^t - x^*} + \frac{\Tilde{L}}{2} \sum_{i=1}^n   \norm{z_i^t - x^*}^2   \right),
    \end{align*}
    where the second inequality follows from the result that if $g: \BR \to \BR^d$ is a continuous function, then $\norm{\int_0^1 g(v) d v} \leq \int_0^1 \norm{g(v)} d v$, and the fourth inequality follows the assumption that the Hessian of $f_i$ is $\tilde{L}$-Lipschitz.
\end{proof}

\subsection{Proof of Lemma \ref{lemma:core_recur}}\label{appendix:core_recur}
We present and prove a more generalized version of Lemma \ref{lemma:core_recur}:
\begin{lemma}
    We initialize each Hessian approximation $B_i^0 = (1 + M \sqrt{L} r_0)^2 E_i^0$ where $E_i^0$ is some PSD matrix that satisfies $E_i^0 \succeq \nabla^2 f_i(x^0)$. 
    For any $\rho$ such that $0 < \rho < 1 - \frac{k}{d}$, there exists positive constants $r_0$ and $\sigma_0$ such that if $\norm{x^0 - x^*} \leq r_0$ and $\nu(E_i^0, \nabla^2 f_i (x^0)) \leq \sigma_0$, for all $i \in [n]$, the sequence of iterates generated by the LISR-$k$ satisfies
    \begin{equation}
        \norm{x^{t+1} - x^*} \leq \rho^{\ceil{\frac{t+1}{n}}} \norm{x^0 - x^*}.
        \label{eq:sr_residual_recur_1}
    \end{equation}
    Furthermore, it holds that
    \begin{equation}
         \nu((\omega^{t+1})^{-1} B_{i_t}^{t+1}, \nabla^2 f_{i_t}(z_{i_t}^{t+1})) \leq \left(1 - \frac{k}{d} \right)^{\ceil{\frac{t+1}{n}}} \delta,
        \label{eq:sr_hessian_recur_2}
    \end{equation}
    where $\delta \coloneqq e^{\frac{4M \sqrt{L} r_0}{1 - \rho}}\big(\sigma_0 + r_0 \frac{4M d L^{\frac{3}{2}} \mu^{-1}}{1 - (1- k d^{-1})^{-1} \rho}\big)$, $M = \Tilde{L} / \mu^{\frac{3}{2}}$, $\omega^t = (1 + \alpha_{\ceil{t/n}})^2$  if $t$ is a multiple of $n$ and 1 otherwise, and the sequence $\{ \alpha_k\}$ is defined as $\alpha_k \coloneqq M \sqrt{L} r_0 \rho^k$, $\forall k \geq 0$.
    \label{lemma:srk_core_recur}
\end{lemma}

\begin{proof}
    Define $c \coloneqq 1 - k d^{-1}$.

    For a given $\rho$ that satisfies $0 < \rho < c$, we choose $r_0, \delta > 0$ such that they satisfy
    \begin{equation}
          \frac{\Tilde{L} \mu^{-1} r_0 }{ 2 } +  (1 + M \sqrt{L} r_0 )^2  \delta  +  M L^{\frac{3}{2}} \mu^{-1} r_0  (2 + M \sqrt{L}  r_0 ) \leq \frac{\rho}{ 1 + \rho}.
         \label{rho_condition}
    \end{equation}
    \textbf{Base case:} At $t=1$, from Lemma \ref{lemma:cur_iter_bound}, we have
    \begin{equation*}
        \norm{x^1 - x^*} \leq \frac{\Tilde{L}\Gamma^0}{2} \sum_{i=1}^n \norm{z_i^0 - x^*}^2 + \Gamma^0 \sum_{i=1}^n \norm{B_i^0 - \nabla^2 f_i (z_i^0)} \norm{z_i^0 - x^*}.
    \end{equation*}
    For the initialization, we have that $B_i^0 = (1 + \alpha_0)^2 E_i^0$ and $z_i^0 = x^0$, for all $i \in [n]$, and $\norm{x^0 - x^*} \leq r_0$. Substituting these in the above expression, we obtain:
    \begin{align*}
        \norm{x^1 - x^*} \leq & \Gamma^0 \left(n\frac{\Tilde{L} r_0}{2} + \sum_{i=1}^n \norm{(1 + \alpha_0)^2 E_i^0 - \nabla^2 f_i (z_i^0)}\right) \norm{x^0 - x^*} \\
        \leq & \Gamma^0 \left(n\frac{\Tilde{L} r_0}{2} + (1 + \alpha_0)^2 \sum_{i=1}^n \norm{ E_i^0 - \nabla^2 f_i (z_i^0)} + \alpha_0 (\alpha_0 + 2) \sum_{i=1}^n \norm{\nabla^2 f_i(z_i^0)}\right) \norm{x^0 - x^*} \\
         \leq & \Gamma^0 \left(n\frac{\Tilde{L} r_0}{2} + (1 + \alpha_0)^2 \sum_{i=1}^n \frac{d L \tau(E_i^0, \nabla^2 f_i(x^0))}{\mathrm{tr}(\nabla^2 f_i(x^0))} + n L \alpha_0 (\alpha_0 + 2) \right) \norm{x^0 - x^*} \\
         \leq & \Gamma^0 \left(n\frac{\Tilde{L} r_0}{2} + (1 + M \sqrt{L} r_0)^2 n \mu \sigma_0 + n M L^{\frac{3}{2}} r_0 (M \sqrt{L} r_0 + 2) \right) \norm{x^0 - x^*} \\
         \leq & \Gamma^0 \left(n\frac{\Tilde{L} r_0}{2} + (1 + M \sqrt{L} r_0)^2 n \mu \delta + n M L^{\frac{3}{2}} r_0 (M \sqrt{L} r_0 + 2) \right) \norm{x^0 - x^*}.
    \end{align*} 
    The second inequality follows from the triangle inequality and the third inequality is due to Lemma \ref{lemma:norm_bound} and the initialization condition $\tau(E_i^0, \nabla^2 f_i(z_i^0)) \leq \sigma_0$.

    We now upper bound $\Gamma^0$. Define $X^0 \coloneqq \frac{1}{n} \sum_{i=1}^n B_i^0$ and $Y^0 \coloneqq \frac{1}{n} \sum_{i=1}^n \nabla^2 f_i (z_i^0)$. Then we have
    \begin{equation*}
        \frac{1}{n} \sum_{i=1}^n \norm{B_i^0 - \nabla^2 f_i (z_i^0)} \geq \norm{X^0 - Y^0} = \norm{(Y^0)((Y^0)^{-1} X^0 - I)} \geq \mu \norm{(Y^0)^{-1} X^0 - I},
    \end{equation*}
    where the second inequality follows from each $f_i$ is $\mu$-strongly convex. By tracking the steps for deriving the bound of $\norm{x^1 - x^*}$, we have
    \begin{equation*}
        \sum_{i=1}^n \norm{B_i^0 - \nabla^2 f_i (z_i^0)} \leq (1 + M \sqrt{L} r_0)^2 n \mu \delta + n M L^{\frac{3}{2}} r_0 (M \sqrt{L} r_0 + 2).
    \end{equation*}
    Combining the above two inequalities, we obtain
    \begin{equation*}
        \norm{(Y^0)^{-1} X^0 - I} \leq (1 + M \sqrt{L} r_0)^2  \delta +  \frac{M L^{\frac{3}{2}} r_0}{\mu} (M \sqrt{L} r_0 + 2) < \frac{\rho}{1 + \rho}.
    \end{equation*}
    We can now upper bound $\Gamma^0$ using Lemma \ref{lemma:banach}. Since the matrix $(Y^0)^{-1}X^0 - I \succ 0$, we can deduce that
    \begin{equation*}
        \norm{(X^0)^{-1} Y^0} = \norm{(I + ((Y^0)^{-1} X^0)- I)^{-1}} \leq \frac{1}{1 - \norm{(Y^0)^{-1} X^0 - I}} \leq 1 + \rho.
    \end{equation*}
    Recall that $\mu I \preceq Y^0$. Consequently, we can bound $\Gamma^0$ with
    \begin{equation*}
        \Gamma^0 = \frac{1}{n} \norm{(X^0)^{-1}} \leq \frac{1+\rho}{n \mu}.
    \end{equation*}
    Therefore, using \eqref{rho_condition}, we have
    \begin{equation*}
         \norm{x^1 - x^*} 
         \leq  (1 + \rho) \left(\frac{\Tilde{L} \mu^{-1} r_0}{2} + (1 + M \sqrt{L} r_0)^2  \delta +  M L^{\frac{3}{2}} \mu^{-1} r_0 (M \sqrt{L} r_0 + 2) \right) \norm{x^0 - x^*} \leq \rho \norm{x^0 - x^*}.
    \end{equation*}
    To complete the base step, we now upper bound $\nu(\omega_{1}^{-1} B_{1}^1, \nabla^2 f_{1}(z_{1}^1))$, where $\omega_1 = 1$. Applying Lemma \ref{lemma:srk_recur} with parameters as $T=1$, $\Tilde{x} = x^*$, $P^0 = (1 + \alpha_0)^2 I_1^0 = B_1^0$, we get
    \begin{align*}
          \nu(B^1, \nabla^2 f(x^1))  \leq &  c e^ {\frac{4 M \sqrt{L}  r_0}{1 - \rho} } \left( \nu (B^{0}, \nabla^2 f(x^{0})) +  \frac{4 M d \kappa \sqrt{L} r_0 }{1 - c^{-1} \rho}  \right) \\
          \leq & c e^ {\frac{4 M \sqrt{L}  r_0}{1 - \rho} } \left( \sigma^0 +  \frac{4 M d  L^{\frac{3}{2}} \mu^{-1} r_0 }{1 - c^{-1} \rho}  \right) \\
          \leq & c \delta.
    \end{align*}
    This completes the proof for the base case.
    
    \textbf{Induction Hypothesis:} Let \eqref{eq:linear_recur_1} and \eqref{eq:linear_recur_2} hold for all $t \in [jn + m - 1]$ for some $j \geq 0$ and $0 \leq m < n$.
    
    \textbf{Induction Step:} We then prove that \eqref{eq:linear_recur_1} and \eqref{eq:linear_recur_2} also hold for $t = jn + m$. Recall that the tuples are updated in a deterministic cyclic order, and at the current time $t$, we are in the $j$th cycle and have updated the $m$th tuple. Therefore, it is easy to note that $z_i^{j n + m} = z_i^{j n + i}$, for all $i \in [m]$, and $z_{i}^{j n + m} = z_i^{j n - n + i}$ for all $i \in [n] \setminus [m]$. From the induction hypothesis, we have
    \begin{equation*}
        \norm{z_i^{j n + m} - x^*} \leq 
        \begin{cases}
            \rho^{\ceil{\frac{j n + i}{n}}} \norm{x^0 - x^*}, & i\in[m] \\
            \rho^{\ceil{\frac{(j - 1) n + i}{n}}} \norm{x^0 - x^*}, & i\in[n] \setminus [m]
        \end{cases}
    \end{equation*}

    \paragraph{Upper Bound of $\sum_{i=1}^n \norm{B_i^{j n + m} - \nabla^2 f_i(z_i^{j n + m})}$.}
    We will establish an upper bound on $\sum_{i=1}^n \norm{B_i^{j n + m} - \nabla^2 f_i(z_i^{j n + m})}$. 
    Since $B_{i_t}^t$ is updated in a different manner for $t~{\rm mod}~n \neq 0$ and $t~{\rm mod}~n = 0$, we analyze these two cases separately.

    \textbf{(a)}. $t~{\rm mod}~n \neq 0$

    Since $t = j n + m + 1$, this case is equivalent to considering $0 \leq m < n - 1$. From the structure of the cyclic updates and \eqref{greedy_sr1_update-0}, we can observe that $B_i^{j n + m} = B_i^{j n + i}$, for all $i \in [m]$, $B_i^{j n + m} = (1 + M \sqrt{L} r_0 \rho^j) B_i^{j n - n + i}$, for all $i \in [n-1] \setminus [m]$, and $B_i^{jn + m} = B_i^{j n}$, for $i = n$.

    For all $i \in [m]$, from the induction hypothesis, we have
    \begin{equation*}
        \norm{B_i^{j n + m} - \nabla^2 f_i (z_i^{j n + m})} = \norm{B_i^{j n + i} - \nabla^2 f_i (z_i^{j n + i})} \leq \frac{L d \tau(B_i^{j n + i}, \nabla^2 f_i (z_i^{j n + i}))}{\mathrm{tr}(\nabla^2 f_i (z_i^{j n + i}))} \leq c^{\ceil{\frac{j n + i}{n}}}\mu \delta .
    \end{equation*}
    The first inequality follows from Lemma \ref{lemma:norm_bound}. 

    For all $i \in [n-1] \setminus [m]$, we have:
    \begin{align*}
        \norm{B_i^{j n + m} - \nabla^2 f_i (z_i^{j n + m})} = & \norm{(1 + M\sqrt{L} r_0 \rho^j)^2 B_i^{j n - n + i} - \nabla^2 f_i(x^{j n - n + i})} \\
        \leq & (1 + M \sqrt{L} r_0 \rho^j)^2 \norm{B_i^{j n - n + i} - \nabla^2 f_i (x^{j n - n + i})} \\ 
        & + M \sqrt{L} r_0 \rho^j (2 + M \sqrt{L} r_0 \rho^j) \norm{\nabla^2 f_i(x^{j n - n + i})} \\
        \leq & (1 + M \sqrt{L} r_0 \rho^j)^2 \norm{B_i^{j n - n + i} - \nabla^2 f_i (x^{j n - n + i})} + M L^{\frac{3}{2}} r_0 \rho^j (2 + M \sqrt{L} r_0 \rho^j)  \\
        \leq & (1 + M \sqrt{L} r_0 \rho^j)^2 c^j \mu \delta + M L^{\frac{3}{2}} r_0 \rho^j (2 + M \sqrt{L} r_0 \rho^j).
    \end{align*}
    The first inequality follows from the triangle inequality. The second inequality is due to the assumption that each $f_i$ is $L$-smooth. The last inequality is deduced from the induction hypothesis.
    
    For $i = n$, we have 
    \begin{align*}
        \norm{B_{i}^{j n + m} - \nabla^2 f_i (z_i^{j n + m})} = & \norm{B_i^{j n} - \nabla^2 f_i (z_i^{j n})} \\
         = & \norm{ \omega_{j n} (\omega_{j n}^{-1} B_i^{j n} - \nabla^2 f_i (z_i^{j n})) + (\omega_{j n} - 1) \nabla^2 f_i (z_i^{j n})} \\
         \leq & \omega_{j n} \norm{  \omega_{j n}^{-1} B_i^{j n} - \nabla^2 f_i (z_i^{j n})} + (\omega_{j n} - 1) \norm{ \nabla^2 f_i (z_i^{j n})} \\
         \leq & (1 + M \sqrt{L} r_0 \rho^j)^2 c^j \mu \delta + M L^{\frac{3}{2}} r_0 \rho^j (2 + M \sqrt{L} r_0 \rho^j).
    \end{align*}
    The last inequality follows from the induction hypothesis and the assumption that each $f_i$ is $L$-smooth.

    We can now bound the quantity $\sum_{i=1}^n \norm{B_i^{j n + m} - \nabla^2 f_i(z_i^{j n + m})}$ as follows:
    \begin{equation}
    \begin{split}
        \sum_{i=1}^n \norm{B_i^{j n + m} - \nabla^2 f_i(z_i^{j n + m})} \leq & m c^{\ceil{\frac{j n + i}{n}}}\mu \delta + (n - m) ((1 + M \sqrt{L} r_0 \rho^j)^2 c^j \mu \delta \\ 
        & + M L^{\frac{3}{2}} r_0 \rho^j (2 + M \sqrt{L} r_0 \rho^j)) \\
        \leq & m \mu \delta + (n - m) ((1 + M \sqrt{L} r_0 )^2 \mu \delta  + M L^{\frac{3}{2}} r_0  (2 + M \sqrt{L} r_0 )) \\
        \leq & n (1 + M \sqrt{L} r_0 )^2 \mu \delta + n  M L^{\frac{3}{2}} r_0  (2 + M \sqrt{L} r_0 ).
    \end{split}
    \label{eq:matrix_diff_1}
    \end{equation}
    \textbf{(b)}. $t~{\rm mod}~n = 0$

    Since $t = j n + m + 1$, we can infer that $m = n -1$. Here we have $B_{i}^{j n} = (1 + M \sqrt{L} r_0 \rho^j)^2 B_i^{j n - n + i}$, for all $i \in [n-1]$, and $B_n^{jn}$ would be used as it is.

    Similar to the case $t \mod n \neq 0$, for $i \in [n - 1]$, we get
    \begin{align*}
        \norm{B_i^{j n + m} - \nabla^2 f_i(z_i^{j n + m})} \leq (1 + M \sqrt{L} r_0 \rho^j)^2 \mu \delta c^j + M L^{\frac{3}{2}} r_0 \rho^j (2 + M \sqrt{L} r_0 \rho^j), 
    \end{align*}
    For $i=n$, we have
     \begin{align*}
        \norm{B_i^{j n + m} - \nabla^2 f_i(z_i^{j n + m})} \leq (1 + M \sqrt{L} r_0 \rho^j)^2 \mu \delta c^j + M L^{\frac{3}{2}} r_0 \rho^j (2 + M \sqrt{L} r_0 \rho^j), 
    \end{align*}
    We can now bound the term $\sum_{i=1}^n \norm{B_{i}^{j n + m} - \nabla^2 f_i (z_i^{j n + m})}$ in the following manner:
    \begin{equation}
    \begin{split}
        \sum_{i=1}^n \norm{B_{i}^{j n + m} - \nabla^2 f_i (z_i^{j n + m})} \leq & n (1 + M \sqrt{L} r_0 \rho^j)^2 \mu \delta c^j + n M L^{\frac{3}{2}} r_0 \rho^j (2 + M \sqrt{L} r_0 \rho^j) \\
        \leq & n (1 + M \sqrt{L} r_0 )^2 \mu \delta  + n M L^{\frac{3}{2}} r_0  (2 + M \sqrt{L} r_0 ).
    \end{split}
    \label{eq:matrix_diff_2}
    \end{equation}
    \paragraph{Upper Bound of $\norm{z_{m+1}^{j n + m + 1} - x^ * }$} 
    We now give the bound of $\norm{z_{m+1}^{j n + m + 1} - x^ * }$ using the bound of $\sum_{i=1}^n \norm{B_i^{j n + m} - \nabla^2 f_i (z_i^{j n + m})}$:
    \begin{align*}
        \norm{z_{m+1}^{j n + m + 1} - x^*} \leq &\frac{\Tilde{L} \Gamma^{j n + m}}{2} \sum_{i=1}^n \norm{z_i^{j n + m} - x^*}^2 \\
        & + \Gamma^{j n + m} \sum_{i=1}^n \norm{B_i^{j n + m} - \nabla^2 f_i (z_i^{j n + m})} \norm{z_i^{j n + m} - x^*} \\
        \leq & \Gamma^{j n + m}\left( n \frac{\Tilde{L} r_0 }{ 2} + n (1 + M \sqrt{L} r_0 )^2 \mu \delta  + n M L^{\frac{3}{2}} r_0  (2 + M \sqrt{L} r_0 ) \right) \rho^j \norm{x^0 - x^*}.
    \end{align*}
    The first inequality is due to Lemma \ref{lemma:cur_iter_bound}. The last inequality follows from \eqref{eq:matrix_diff_1}, \eqref{eq:matrix_diff_2}, and the induction hypothesis.

    To bound $\Gamma^{j n + m}$, let $X^{j n + m} \coloneqq \frac{1}{n} \sum_{i=1}^n B_{i}^{j n + m}$ and $Y^{j n + m} \coloneqq \frac{1}{n} \sum_{i=1}^n \nabla^2 f_i (z_i^{j n + m})$. Observe that:
    \begin{equation*}
        \frac{1}{n} \sum_{i=1}^{n} \norm{B_i^{j n + m} - \nabla^2 f_i (z_i^{j n + m})} \geq \norm{X^{j n + m} - Y^{j n + m}} \geq \mu \norm{(Y^{j n + m})^{-1} X^{j n + m} - I}.
    \end{equation*}
    The last inequality follows from the assumption that each $f_i$ is $\mu$-strongly convex. Furthermore, using the \eqref{eq:matrix_diff_1} and \eqref{eq:matrix_diff_2}, we can see that:
    \begin{equation*}
         \frac{1}{n} \sum_{i=1}^{n} \norm{B_i^{j n + m} - \nabla^2 f_i (z_i^{j n + m})} \leq   (1 + M \sqrt{L} r_0 )^2 \mu \delta  +  M L^{\frac{3}{2}} r_0  (2 + M \sqrt{L} r_0 ).
    \end{equation*}
    Combining the above two inequalities, we can infer that:
    \begin{equation*}
        \norm{(Y^{j n + m})^{-1} X^{j n + m} - I} \leq  (1 + M \sqrt{L} r_0 )^2 \delta  +  M L^{\frac{3}{2}} \mu^{-1} r_0  (2 + M \sqrt{L} r_0 ) \leq \frac{\rho}{1 + \rho}.
    \end{equation*}
    We can now bound $\Gamma^{j n + m}$ using Banach's Lemma following the same way we derive for the base case. We get that 
    \begin{equation*}
        n \Gamma^{j n + m} =  \norm{(X^{j n + m})^{-1}} \leq \frac{1 + \rho}{\mu}.
    \end{equation*}
    Consequently, we have
    \begin{equation}
          \begin{split}
         \norm{z_{m+1}^{j n + m + 1} - x^*} 
        \leq & \Gamma^{j n + m}\left( n \frac{\Tilde{L} r_0 }{ 2} + n (1 + M \sqrt{L} r_0 )^2 \mu \delta  + n M L^{\frac{3}{2}} r_0  (2 + M \sqrt{L} r_0 ) \right) \rho^j \norm{x^0 - x^*} \\
        \leq & \frac{1 + \rho}{n \mu} \left( n \frac{\Tilde{L} r_0 }{ 2} + n (1 + M \sqrt{L} r_0 )^2 \mu \delta  + n M L^{\frac{3}{2}} r_0  (2 + M \sqrt{L} r_0 ) \right) \rho^j \norm{x^0 - x^*} \\
        \leq & \rho^{j + 1} \norm{x^0 - x^*}.
    \end{split}
    \label{stage_2_upper_bound}
    \end{equation}
    The last inequality follows from \eqref{rho_condition}.

    \paragraph{Upper Bound of $\nu(\omega_{j n + m + 1}^{-1} B_{m + 1}^{j n + m + 1}, \nabla^2 f_{m + 1} (z_{m+1}^{j n + m + 1}))$}
    To provide the upper bound of the given metric, our first step is  to establish that $\tau(\omega_{j n + m + 1}^{-1} B_{m + 1}^{j n + m + 1}, \nabla^2 f_{m + 1} (z_{m+1}^{j n + m + 1}))$ is well-defined, by showing that:
    \begin{equation}
        \omega_{j n + m + 1}^{-1} B_{m + 1}^{j n + m + 1} \succeq \nabla^2 f_{m + 1} (z_{m+1}^{j n + m + 1}).
        \label{eq:valid_metric}
    \end{equation}
    To establish the above relation, we need two observations. The first relation is that 
    \begin{equation*}
        \omega_{j n + m}^{-1} B_{m + 1}^{j n + m} \succeq \nabla^2 f_{m+1}(z_{m+1}^{j n + m}) \succeq \left(1 + \frac{M r_{j n + m + 1}}{2}\right)^{-1} K^{j n + m}, 
    \end{equation*}
    where the first inequality follows from the induction hypothesis and the second inequality follows from Lemma \ref{lemma:matrix_approx}.

    In addition, using the deductions from Lemma \ref{lemma:srk_recur}, we have:
    \begin{equation}
        r_{j n + m + 1} \leq \frac{2 \alpha_j}{M}.
        \label{r_bound}
    \end{equation}
    Now we consider two cases depending on $m$.

    \textbf{(a)}: $0 \leq m < n - 1$
    Since all the $B_i$ were scaled by a factor $(1 + \alpha_j)^2$ at the end of the cycle $j - 1$, i.e., at $t = j n$, we have $B_{m+1}^{j n + m} = (1 + \alpha_j)^2 B_{m+1}^{(j - 1) n + m + 1}$. Furthermore, from the induction hypothesis, we have:
    \begin{equation*}
        \omega_{((j - 1)n + m + 1)_{+}}^{-1} B_{m+1}^{((j - 1) n + m + 1)_{+}} \succeq \nabla^2 f_{m+1}(z_{m+1}^{((j - 1)n + m + 1)_{+}}).
    \end{equation*}
    where we define $(x)_{+} = \max(x, 0)$.

    Since $0 \leq m < n - 1$, we have $\omega_{(j - 1)n + m + 1} = 1$. Also, $z_{m+1}^{j n + m} = z_{m+1}^{((j - 1)n + m + 1)_{+}}$. Therefore,
    \begin{align*}
        B_{m+1}^{j n + m} = (1 + \alpha_j)^2 B_{m+1}^{((j - 1)n + m + 1)_{+}} \succeq & (1 + \alpha_j)^2 \nabla^2 f_{m+1}(z_{m+1}^{j n + m}) \\
        \succeq & (1 + \alpha_j)^2 \left(1 + \frac{M r_{j n + m + 1}}{2}\right)^{-1} K^{j n + m + 1} \\
        \succeq & (1 + \alpha_j) K^{j n + m + 1}.
    \end{align*}
    The second inequality follows from Lemma \ref{lemma:matrix_approx}. The last inequality can be inferred from \eqref{r_bound}.

    \textbf{(b)}. $m = n - 1$

    Since the current index $m + 1$ was last updated at time $t = j n$, we have $B_{m + 1}^{j n + m} = B_{m + 1}^{j n}$ and $z_{m + 1}^{j n + m} = z_{m + 1}^{j n}$. Further, the induction hypothesis yields that $\omega_{j n}^{-1} B_{m+1}^{j n} \succeq \nabla^2 f_{m+1} (z_{m + 1}^{j n})$. Also, $\omega_{j n} = (1 + \alpha_j)^2$ by definition. Therefore,
    \begin{align*}
        B_{m+1}^{j n + m} = \omega_{j n} (\omega_{j n}^{-1} B_{m + 1}^{j n}) \succeq& (1 + \alpha_j)^2 \nabla^2 f_{m+1}(z_{m+1}^{j n + m}) \\
        \succeq & (1 + \alpha_j)^2 (1 + \frac{M r_{j n + m + 1}}{2})^{-1} K^{j n + m + 1} \\
        \succeq & (1 + \alpha_j) K^{j n + m + 1}.
    \end{align*}
    To sum up, for both cases $0 \leq m < n$ and $m = n - 1$, we have established that $B_{m+1}^{j n + m} \succeq (1 + \alpha_j) K^{j n + m + 1}$.

    Applying Lemma \ref{lemma:matrix_approx} to $B_{m+1}^{j n + m} \succeq (1 + \alpha_j)K^{j n + m + 1}$, we have:
    \begin{align*}
        B_{m+1}^{j n + m} \succeq (1 + \alpha_j) K^{j n + m + 1} \succeq \left(1 + \frac{M r_{j n + m + 1}}{2} \right) K^{j n + m + 1} \succeq \nabla^2 f_{m+1}(z_{m+1}^{j n + m + 1}).
    \end{align*}
    For Lemma \ref{lemma:core_recur}, applying Lemma \ref{lemma:matrix_approx} and Lemma \ref{lemma:sr1_one_step}, we have:
    \begin{align*}
        \omega_{j n + m + 1}^{-1} B_{m + 1}^{j n + m + 1} = & \mathrm{SR1} (B_{m+1}^{j n + m}, \nabla^2 f_{m+1} (z_{m+1}^{j n + m + 1}), \Bar{u}(B_{m+1}^{j n + m}, \nabla^2 f_{m+1} (z_{m+1}^{j n + m + 1}) ) \\
        \succeq & \nabla^2 f_{m+1} (z_{m+1}^{j n + m + 1}).
    \end{align*}

    For Lemma \ref{lemma:srk_core_recur}, applying Lemma \ref{lemma:matrix_approx} and Lemma \ref{lemma:srk_update_consist}, it follows that:
    \begin{align*}
        \omega_{j n + m + 1}^{-1} B_{m + 1}^{j n + m + 1} = & \mathrm{SR}\textrm{-}k (B_{m+1}^{j n + m}, \nabla^2 f_{m+1} (z_{m+1}^{j n + m + 1}), \Bar{U}(B_{m+1}^{j n + m}, \nabla^2 f_{m+1} (z_{m+1}^{j n + m + 1}) ) \\
        \succeq & \nabla^2 f_{m+1} (z_{m+1}^{j n + m + 1}).
    \end{align*}

    Now define the sequence $\{y_k\}$, for $k = 0, \dots, j + 1$, such that $\{y_k\} = \{x^0, z_{m+1}^{m+1}, \dots, z_{m+1}^{j n + m + 1} \}$. 
    From the induction hypothesis and the upper bound of (\ref{stage_2_upper_bound}), we can infer that $\norm{y_k - x^*} \leq \rho^k \norm{x^0 - x^*}$, for all $k$. 
    Since $\{y_k\}$ comes from the application of SR1 updates (SR-$k$ updates), the sequence satisfies the conditions (\ref{eq:dist_hypo_1}) from Lemma \ref{lemma:srk_recur}, this implies that
    \begin{align*}
           \nu(\omega_{j n + m + 1}^{-1} B_{m+1}^{j n + m + 1}, \nabla^2 f_{m+1}(y_{j+1}) \leq c^{j+1} \delta.
    \end{align*}
    Since $y_{j + 1} = z_{m+1}^{j n + m + 1}$, this completes the proof.
\end{proof}

\subsection{Proof of Lemma \ref{lemma:sr1_avg}}
We present and prove a more generalized version of Lemma \ref{lemma:sr1_avg}:
\begin{lemma}
    The sequence of iterates generated by the LISR-$k$ method satisfies
    \begin{equation*}
        \norm{x^{t+1} - x^*} \leq \left(1 - \frac{k}{d}\right)^{\ceil{\frac{t+1}{n}}}\frac{1}{n}\sum_{i=1}^n \norm{x^{t+1-i} - x^*}.
    \end{equation*}
    \label{lemma:srk_avg}
\end{lemma}
\begin{proof}
    Define $c \coloneqq 1 - k / d$.

    We prove the Lemma for a generic iteration $t = j n+ m + 1$, for some $j \geq 0$ and $0 \leq m < n$. We restate a few observations derived in the proof of Lemma \ref{lemma:core_recur}. First, the upper bound of $\Gamma^{j n + m} = \norm{(\sum_{i=1}^n B_i^{j n + m})^{-1}}$ is:
    \begin{equation}
        \Gamma^{j n + m} \leq \frac{1 + \rho}{n \mu}.
        \label{prev_result_1}
    \end{equation}
    Next, we recall the upper bound of $\norm{B_{i}^{j n + m} - \nabla^2 f_i (z_i^{j n + m})}$ is:
    \begin{align*}
        \norm{B_i^{j n + m} - \nabla^2 f_i(z_i^{j n + m})} \leq 
        \begin{cases}
        c^{j + 1} \mu \delta, & i \in [m]\\    
        (1 + M \sqrt{L}r_0 \rho^j)^2 c^j \mu \delta + M L^{\frac{3}{2}} r_0 \rho^j (2 + M \sqrt{L} r_0 \rho^j), & i \in [n]\setminus [m]
        \end{cases}
    \end{align*}
    Both cases can be summarized with a common upper bound:
    \begin{equation}
         \norm{B_i^{j n + m} - \nabla^2 f_i(z_i^{j n + m})} \leq (1 + M \sqrt{L}r_0 \rho^j)^2 c^j \mu \delta + M L^{\frac{3}{2}} r_0 \rho^j (2 + M \sqrt{L} r_0 \rho^j), \quad i \in [n].
        \label{prev_result_2}
    \end{equation}
    Finally, we can also establish that
    \begin{equation}
        \begin{split}
        \norm{z_i^{j n + m} - x^*} \leq & \rho^{j + 1} \norm{x^0 - x^*} \leq \rho^{j } \norm{x^0 - x^*}, \quad i \in [m] \\
        \norm{z_i^{j n + m} - x^*} \leq & \rho^{j } \norm{x^0 - x^*}, \quad i \in [n] \setminus [m]
    \end{split}
    \label{prev_result_3}
    \end{equation}

    From Lemma \ref{lemma:cur_iter_bound}, we have
    \begin{align*}
        \norm{z_{m + 1}^{j n + m + 1} - x^*} \leq &\Gamma^{j n + m} \frac{\Tilde{L}}{2} \sum_{i = 1}^{n} \norm{z_i^{j n + m} - x^*}^2 + \Gamma^{j n + m} \sum_{i=1}^n \norm{B_i^{j n + m} - \nabla^2 f_i (z_i^{j n + m})} \norm{z_i^{j n + m} - x^*} \\
        \leq & \Gamma^{j n + m} \sum_{i = 1}^{n} \left(\frac{\Tilde{L}}{2}  \norm{z_i^{j n + m} - x^*} +  \norm{B_i^{j n + m} - \nabla^2 f_i (z_i^{j n + m})}\right) \norm{z_i^{j n + m} - x^*} \\
        \leq & \frac{1 + \rho}{n \mu}  \left(\frac{\Tilde{L} \rho^j r_0}{2}   + (1 + M \sqrt{L}r_0 \rho^j)^2 c^j \mu \delta + M L^{\frac{3}{2}} r_0 \rho^j (2 + M \sqrt{L} r_0 \rho^j) \right) \sum_{i = 1}^{n} \norm{z_i^{j n + m} - x^*} \\
        \overset{\rho < 1}{\leq} & \frac{1 + \rho}{n \mu}  \left(\frac{\Tilde{L} \rho^j r_0}{2}   + (1 + M \sqrt{L}r_0)^2 c^j \mu \delta + M L^{\frac{3}{2}} r_0 \rho^j (2 + M \sqrt{L} r_0 ) \right) \sum_{i = 1}^{n} \norm{z_i^{j n + m} - x^*} \\
         \leq & c^j (1 + \rho)  \left(\frac{\Tilde{L} \mu^{-1} r_0}{2}   + (1 + M \sqrt{L}r_0)^2  \delta + M L^{\frac{3}{2}} \mu^{-1} r_0  (2 + M \sqrt{L} r_0 ) \right) \frac{1}{n}\sum_{i = 1}^{n} \norm{z_i^{j n + m} - x^*} \\
         \overset{(\ref{rho_condition})}{\leq} & c^{j} \rho (\frac{1}{n} \sum_{i=1}^n \norm{z_i^{j n + m} - x^*}) \\
         \leq & c^{j + 1} (\frac{1}{n} \sum_{i=1}^n \norm{x^{t - i} - x^*}).
    \end{align*}
    The third inequality is due to equation (\ref{prev_result_1}), (\ref{prev_result_2}) and (\ref{prev_result_3}). The fifth inequality follows from $\rho < 1 - d^{-1}$ in the condition of Lemma \ref{lemma:core_recur} for Lemma \ref{lemma:sr1_avg} and $\rho < 1 - k d^{-1}$ in the condition of Lemma \ref{lemma:srk_core_recur} for Lemma \ref{lemma:srk_avg}.
\end{proof}

\subsection{Proof of Theorem \ref{thm:glins_res} and Theorem \ref{thm:glins+_res}}
\begin{proof}
    The proof is the same as Theorem 1 by \cite{lahoti2023sharpened} with $c \coloneqq \mu d^{-1} L^{-1}$ replaced by $c \coloneqq d^{-1}$ for Theorem \ref{thm:glins_res} and $c \coloneqq k d^{-1}$ for Theorem \ref{thm:glins+_res}.
\end{proof}

\begin{algorithm}[tb]
\caption{The LISR-1 method}
\label{alg:iqn_algos_efficient}
\textbf{Input}: Initial point $x^0$, a predefined sequence $\{\alpha_k\}_{k=0}^{\infty}$, and a sequence of initial matrices $\{E_i^0\}_{i=1}^n$ s.t. $E_i^0 \succeq \nabla^2 f_i(z_i^0)$, $\forall i \in [n]$.\\
\textbf{Output}: The last iterate $x^t$.
\begin{algorithmic}[1] 
\STATE Initialize $z_i^0 = x^0$, $\forall i \in [n]$; Initialize $B_i^0 = (1 + \alpha_0)^2 E_i^0$, $\forall i \in [n]$;
\STATE Maintain running auxiliary variables $x$, $\Bar{u}$.
\STATE Let $t=0$.
\WHILE{not converged}
\STATE Current index to be updated is $i_t \gets t \mod n$;
\STATE Update $x^t$ as $x \gets (\Bar{B})^{-1} (\phi - g)$.
\STATE Compute $\omega_t$;// $\omega_t = (1 + \alpha_{\ceil{t/n}})^2$ if $t \mod n = 0$ and 1 otherwise
\STATE Update $\Bar{u}$ as $\Bar{u} \gets \argmax_{u \in \{ e_i\}_{i=1}^d} (\langle u, (1 + \alpha_{\ceil{t/n}})^2 B_{i_t} u \rangle - \langle u, \nabla^2 f_{i_t}(x^t) u\rangle)$;
\STATE Update $Q \gets  \mathrm{SR}1 ((1 + \alpha_{\ceil{t/n}})^2 B_{i_t}, \nabla^2 f_{i_t}(x^t), \Bar{u})$;
\STATE Update $B^{\textrm{old}} \gets B_{i_t}$;
\STATE Update $B_{i_t} \gets Q$;
\STATE Update $z_{i_t}^t$ as $z_{i_t}^t = x^t$.
\STATE Update $\phi$ as $\phi \gets \omega_t (\phi - B^{\textrm{old}} z_{i_t}) + B_{i_t} x^t$;
\STATE Update $g$ as $g \gets g + (\nabla f_{i_t} (x^t) - \nabla f_{i_t} (z_{i_t}))$;
\STATE Update $(\Bar{B})^{-1}$ as $(\Bar{B})^{-1} \gets \omega_t^{-1} \mathrm{Sherman}\textrm{-} \mathrm{Morrison}(\Bar{B}^{-1}, v, v, \langle \Bar{u}, v \rangle)$; // $v = ((1 + \alpha_{\ceil{t/n}})^2 B_{i_t} - \nabla^2 f_{i_t}(z_{i_t})) \Bar{u}$.
\STATE Increment the iteration counter $t$.
\ENDWHILE
\STATE \textbf{return} $x^t$.
\end{algorithmic}
\end{algorithm}

\begin{algorithm}[H]
\caption{The LISR-$k$ method}
\label{alg:glins+_efficient}
\textbf{Input}: Initial point $x^0$, a predefined sequence $\{\alpha_k\}_{k=0}^{\infty}$, and a sequence of initial matrices $\{E_i^0\}_{i=1}^n$ s.t. $E_i^0 \succeq \nabla^2 f_i(z_i^0)$, $\forall i \in [n]$.\\
\textbf{Output}: The last iterate $x^t$.
\begin{algorithmic}[1] 
\STATE Initialize $z_i^0 = x^0$, $\forall i \in [n]$; Initialize $B_i^0 = (1 + \alpha_0)^2 E_i^0$, $\forall i \in [n]$;
\STATE Maintain running auxiliary variables $x$, $\Bar{u}$.
\STATE Let $t=0$.
\WHILE{not converged}
\STATE Current index to be updated is $i_t \gets t \mod n$;
\STATE Update $x^t$ as $x \gets (\Bar{B})^{-1} (\phi - g)$.
\STATE Compute $\omega_t$;// $\omega_t = (1 + \alpha_{\ceil{t/n}})^2$ if $t \mod n = 0$ and 1 otherwise
\STATE Update $\Bar{U}$ as $\Bar{U} \gets E_k ((1 + \alpha_{\ceil{t/n}})^2 B_{i_t} - \nabla^2 f_{i_t}(x^t) )$;
\STATE Update $Q \gets  \mathrm{SR}\textrm{-}k ((1 + \alpha_{\ceil{t/n}})^2 B_{i_t}, \nabla^2 f_{i_t}(x^t), \Bar{U})$;
\STATE Update $B^{\textrm{old}} \gets B_{i_t}$;
\STATE Update $B_{i_t} \gets Q$;
\STATE Update $z_{i_t}^t$ as $z_{i_t}^t = x^t$.
\STATE Update $\phi$ as $\phi \gets \omega_t (\phi - B^{\textrm{old}} z_{i_t}) + B_{i_t} x^t$;
\STATE Update $g$ as $g \gets g + (\nabla f_{i_t} (x^t) - \nabla f_{i_t} (z_{i_t}))$;
\STATE Update $(\Bar{B})^{-1}$ as $(\Bar{B})^{-1} \gets \omega_t^{-1} \mathrm{Sherman}\textrm{-} \mathrm{Morrison}(\Bar{B}^{-1}, V,  V, \Bar{U}^{\top} V )$; // $V = ((1 + \alpha_{\ceil{t/n}})^2 B_{i_t} - \nabla^2 f_{i_t}(z_{i_t})) \Bar{U}$.
\STATE Increment the iteration counter $t$.
\ENDWHILE
\STATE \textbf{return} $x^t$.
\end{algorithmic}
\end{algorithm}

\begin{algorithm}[H]
\caption{Sherman-Morrison($A^{-1}$, U, V, W)}
\label{alg:sherman_morrison}
\textbf{Input}: $A^{-1} \in \BR^{d \times d}$, $U \in \BR^{d \times k}$, $V \in \BR^{d \times k}$, $W \in \BR^{k \times k}$.\\
\textbf{Output}: The updated $A^{-1}$.
\begin{algorithmic}[1] 
\STATE \textbf{return} $A^{-1} + A^{-1}U \left(W - U^{\top}A^{-1} V\right)^{-1} V^{\top} A^{-1}$.
\end{algorithmic}
\end{algorithm}

\section{Efficient Implementation of the LISR-1/LISR-$k$ method}\label{appendix:efficient_impl}
Recall that although Algorithm \ref{alg:iqn_algos} enjoys per-iteration complexity of $\fO(d^2)$, it applies scaling to all Hessian approximations at the iteration $t$ satisfying $t \mod n = 0$.
It results in the total FLOPs of $\fO(n d^2)$ at that iteration.
We can improve the algorithm by carrying out the scaling of $B_i$'s lazily. 
In particular, by only scaling the Hessian approximations when they are used, we can improve the per-iteration complexity of LISR-1 (LISR-$k$) to $\fO(d^2)$ for any iteration. The resulting algorithm for the LISR-1 (LISR-$k$) is specified in the Algorithm \ref{alg:iqn_algos_efficient} (Algorithm \ref{alg:glins+_efficient}). 
For Algorithm \ref{alg:iqn_algos_efficient}, it leverages the Sherman-Morrison Formula (Algorithm \ref{alg:sherman_morrison} with $k=1$). 
On the other hand, Algorithm \ref{alg:glins+_efficient} applies the general Sherman-Morrison Formula (Algorithm \ref{alg:sherman_morrison} with $k > 1$).

\section{Experiment Details}\label{appendix:more_exps}
In this section, we present more details of the experiment. 
First, we give an overview of the datasets used in the logistic regression task. 
Then we study the impact of different hyperparameters on the quadratic function minimization task.
Finally, we present the comparison result of the LISR-$k$ methods with different choices of $k$.
\subsection{Datasets}
We present more details of real-world datasets used in Table \ref{tab:dataset}. 
Specifically, each line shows the name of the dataset, the number of data instances, the dimension of each data, the regularization parameter we chose for the logistic regression objective function (\ref{logistic_obj}), and the resulting conditional number.
All datasets can be found on the LIBSVM website.
We preprocess the datasets by changing $\{0, 1\}$ labels to $\{\pm 1\}$.
\begin{table*}[t]
\centering
\caption{Details of the real-world datasets.}
\label{tab:dataset}
\begin{tabular}{ccccc}
    \toprule
    Dataset & $n$ & $d$ & $\lambda$ & $\kappa$ \\
    \midrule
    a9a  & 32,561 & 123 &  $1 \times 10^{-3}$  &  $4.71 \times 10^{7}$ \\\addlinespace
    w8a  & 49,749 & 300 &  $1 \times 10^{-4}$  &  $2.93 \times 10^{8}$ \\\addlinespace
    ijcnn1  & 49,749 & ~~22 &  $1 \times 10^{-4}$  &  $1.37 \times 10^{7}$ \\\addlinespace
    mushrooms  & ~~8,124 & 112 &  $1 \times 10^{-3}$  &  $2.08 \times 10^{7}$ \\\addlinespace
    phishing  & 11,055 & ~~68 &  $1 \times 10^{-4}$  &  $1.79 \times 10^{7}$ \\\addlinespace
    svmguide3  & ~~1,243 & ~~21 &  $1 \times 10^{-3}$  &  $6.67 \times 10^{5}$ \\\addlinespace
    german.numer  & ~~1,000 & ~~24 &  $1 \times 10^{-3}$  &  $2.11 \times 10^{6}$ \\\addlinespace
     splice  & ~~1,000 & ~~60 &  $1 \times 10^{-4}$  &  $4.34 \times 10^{6}$ \\\addlinespace
      covtype  & 50,000 & ~~54 &  $1 \times 10^{-3}$  &  $3.93 \times 10^{7}$ \\\addlinespace
    \bottomrule
\end{tabular}
\end{table*}
\subsection{Effect of different hyperparameters on the Quadratic Function Minimization Task}
\begin{figure*}[!tb]
\centering
\begin{tabular}{ccc}
     \includegraphics[scale=0.45]{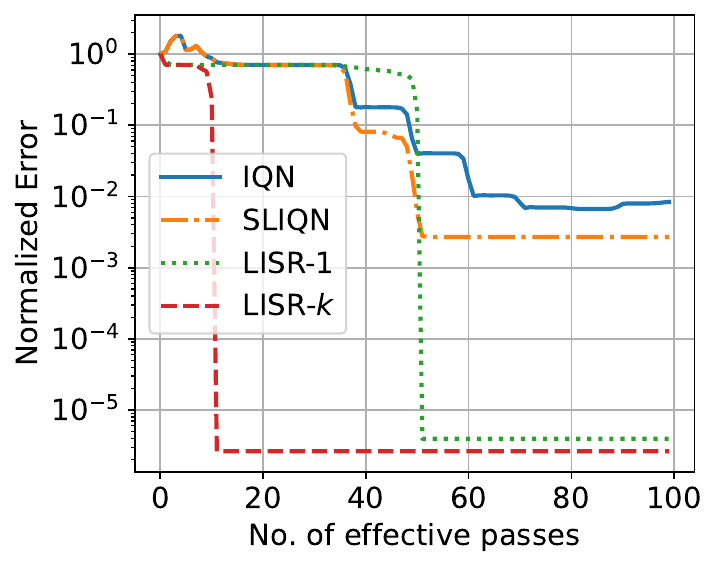}
     & \includegraphics[scale=0.45]{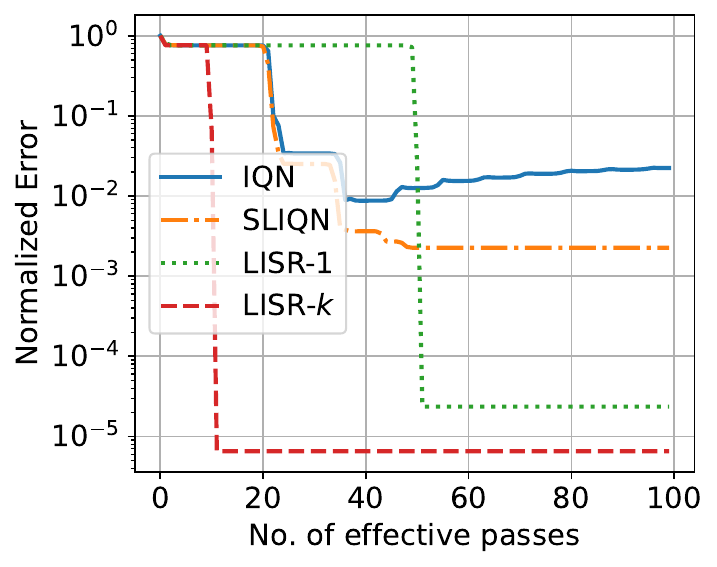}
     & \includegraphics[scale=0.45]{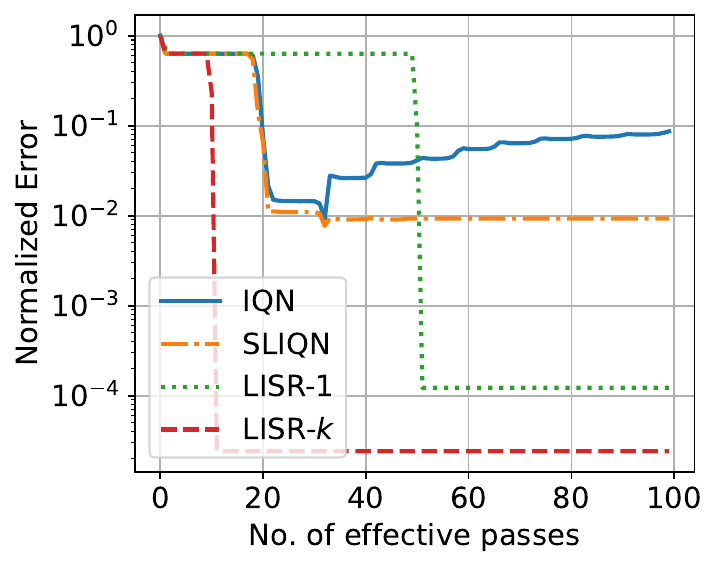} \\
     (a) $n=10$, $d=100$ & (b) $n=100$, $d=100$ &  (c) $n=1000$, $d=100$ \\
     \includegraphics[scale=0.45]{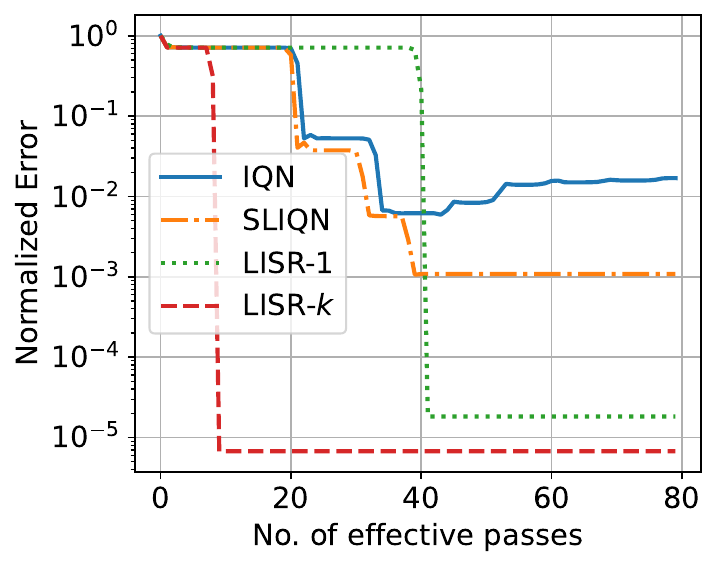}
     & \includegraphics[scale=0.45]{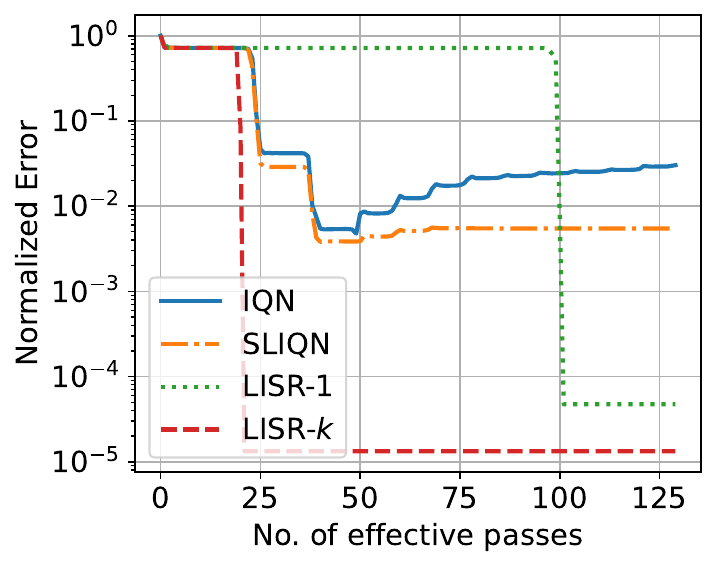}
     & \includegraphics[scale=0.45]{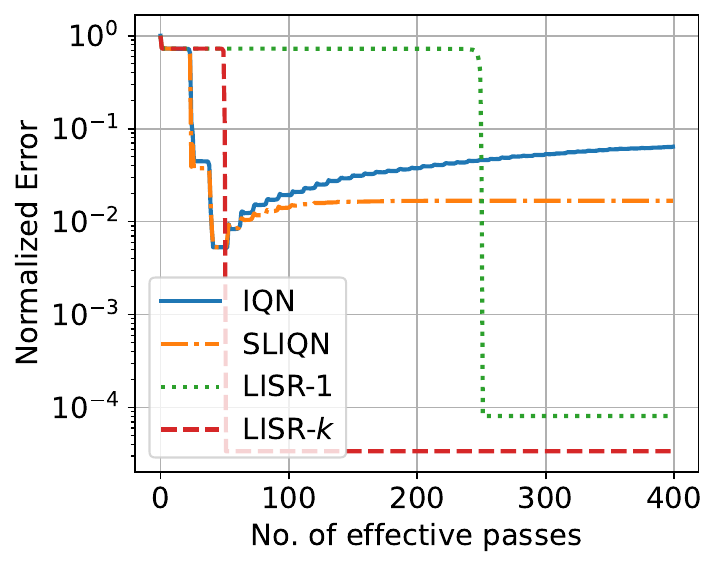} \\   
     (e) $n=100$, $d=80$ & (f) $n=100$, $d=200$ &  (g) $n=100$, $d=500$ \\
\end{tabular}
\caption{Comparison of the proposed methods with baselines for the quadratic function minimization problem. }
\label{fig:quadratic_compare_res}
\end{figure*}

We present the comparison result of our proposed method and baseline methods on the quadratic function minimization task in Figure \ref{fig:quadratic_compare_res}.
It can be observed that tuning different choices of hyperparameters $n$ and $d$ does not have much impact on the superlinear convergence of the LISR-1 and LISR-$k$ methods.
However, the time that superlinear convergence kicks in is correlated with the dimension of the underlying problem. 
Interestingly, the normalized error of the IQN method starts to diverge after reaching certain accuracy.
\subsection{Effect of $k$ on the Logistic Regression Task}
\begin{figure*}[!tb]
\centering
\begin{tabular}{ccc}
     \includegraphics[scale=0.45]{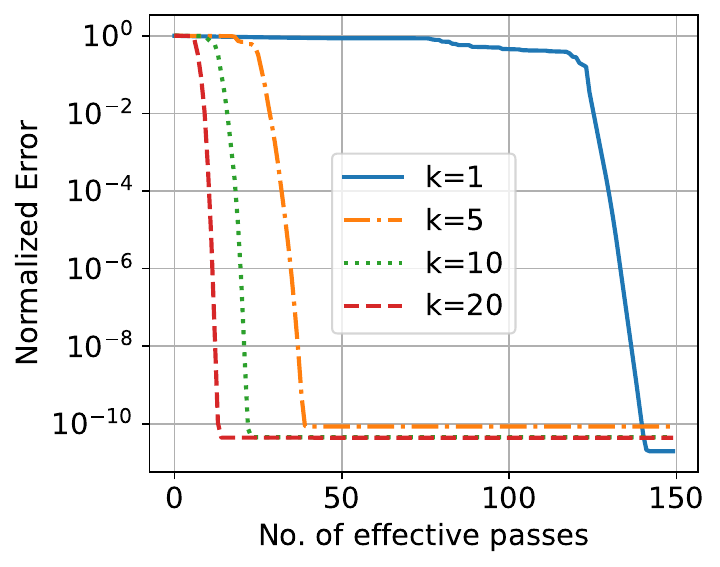}
     & \includegraphics[scale=0.45]{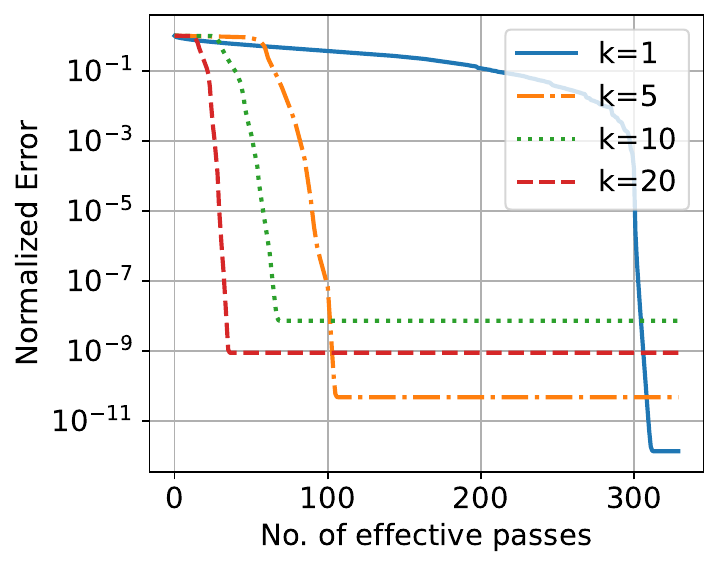}
     & \includegraphics[scale=0.45]{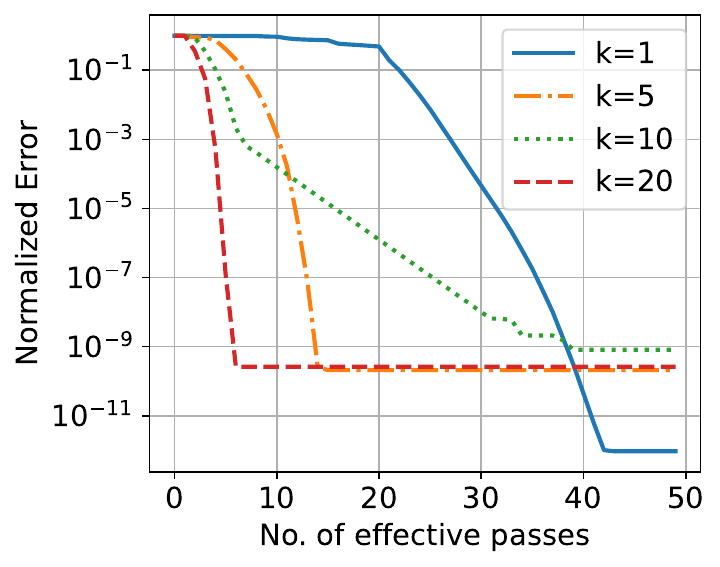} \\
     (a) a9a & (b) w8a &  (c) ijcnn1 \\
     \includegraphics[scale=0.45]{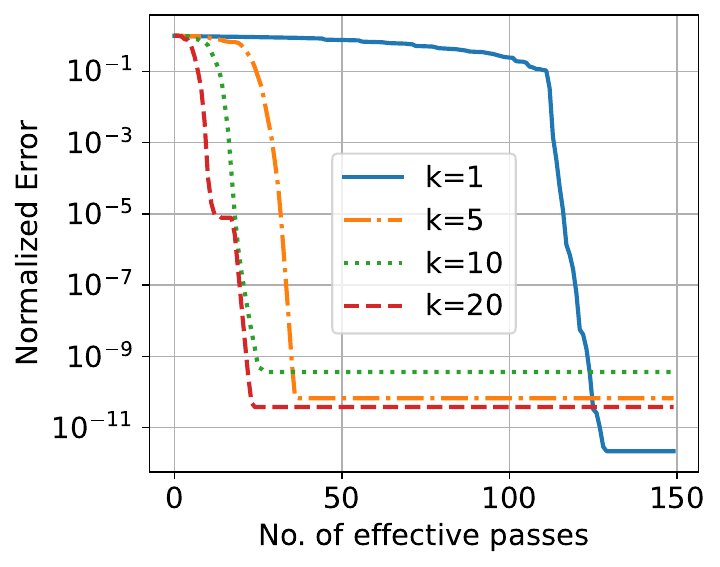}
     & \includegraphics[scale=0.45]{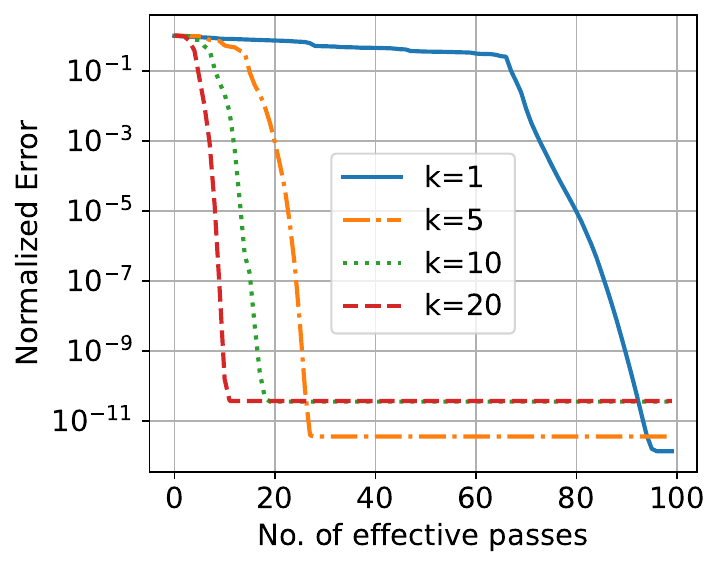}
     & \includegraphics[scale=0.45]{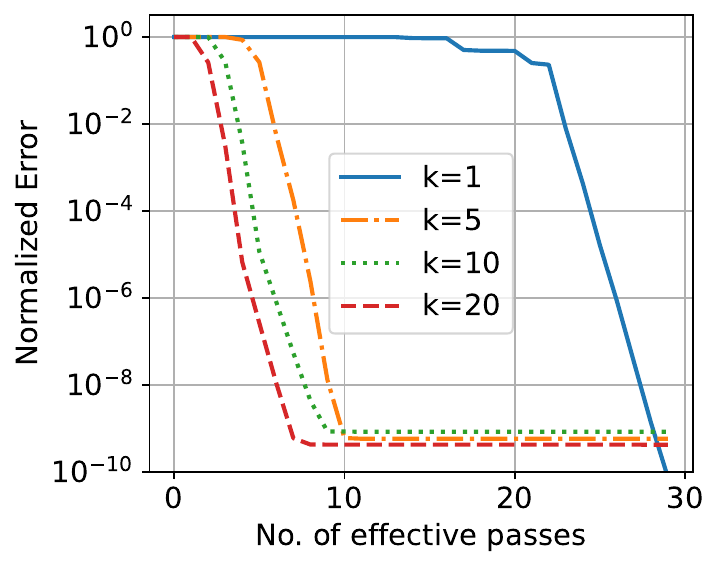} \\   
     (e) mushrooms & (f) phishing &  (g) svmguide3 \\
     \includegraphics[scale=0.45]{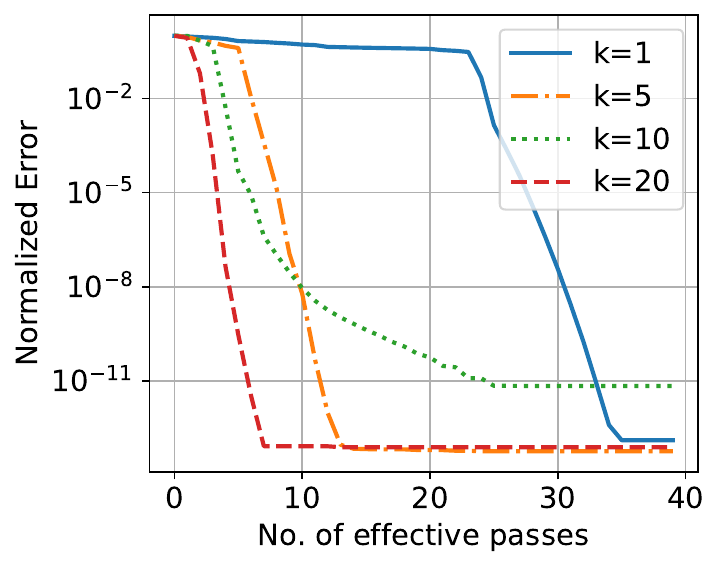}
     & \includegraphics[scale=0.45]{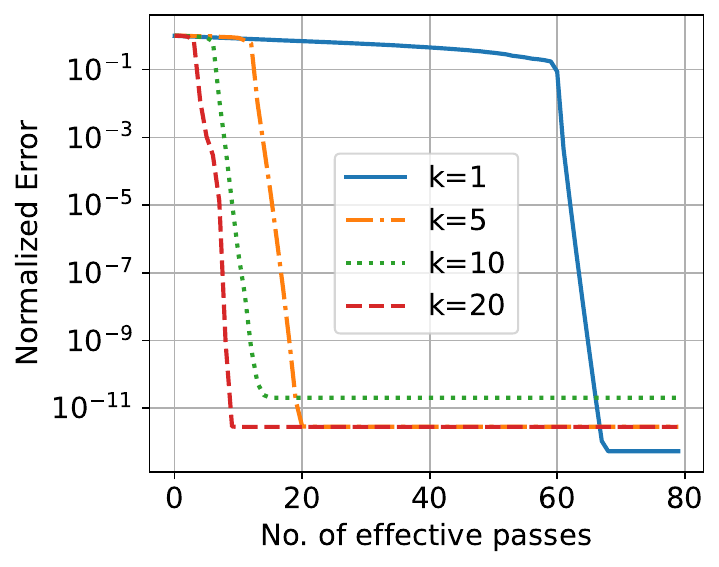}
     & \includegraphics[scale=0.45]{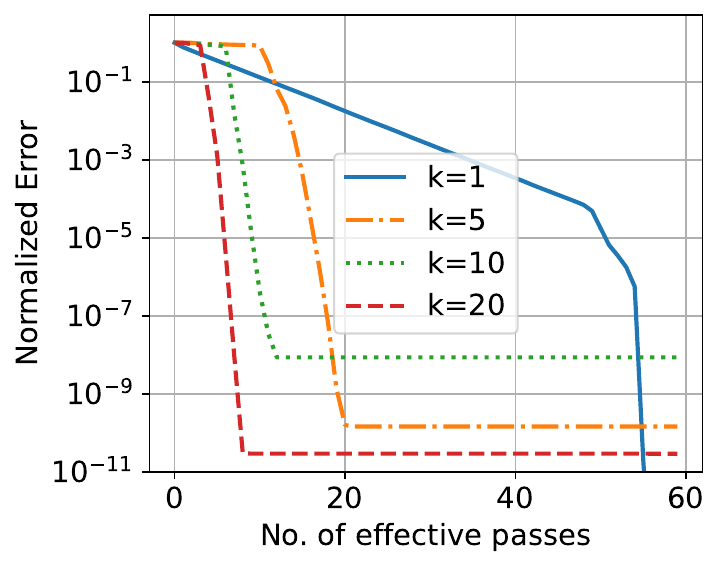}    \\
     (h) german.numer & (i) splice &  (j) covtype \\
\end{tabular}
\caption{Comparison of the LISR-$k$ method with different choices of $k$ for the general function minimization. }
\label{fig:general_compare_res}
\end{figure*}

\begin{figure*}[!tb]
\centering
\begin{tabular}{ccc}
     \includegraphics[scale=0.45]{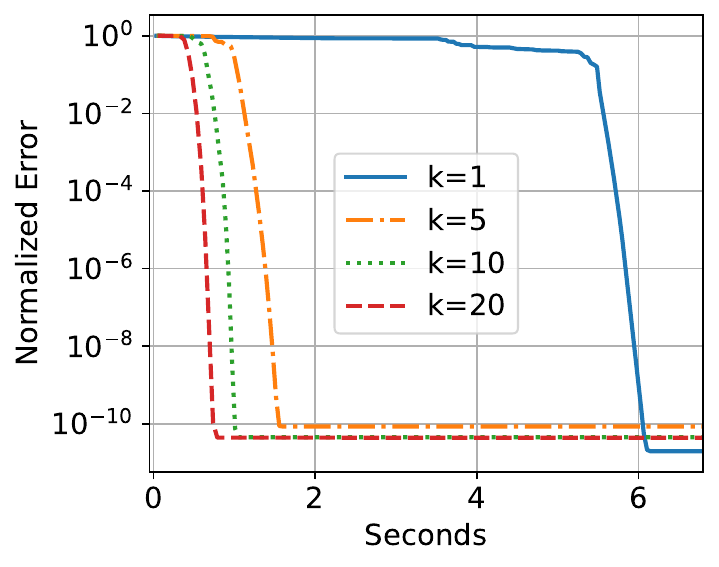}
     & \includegraphics[scale=0.45]{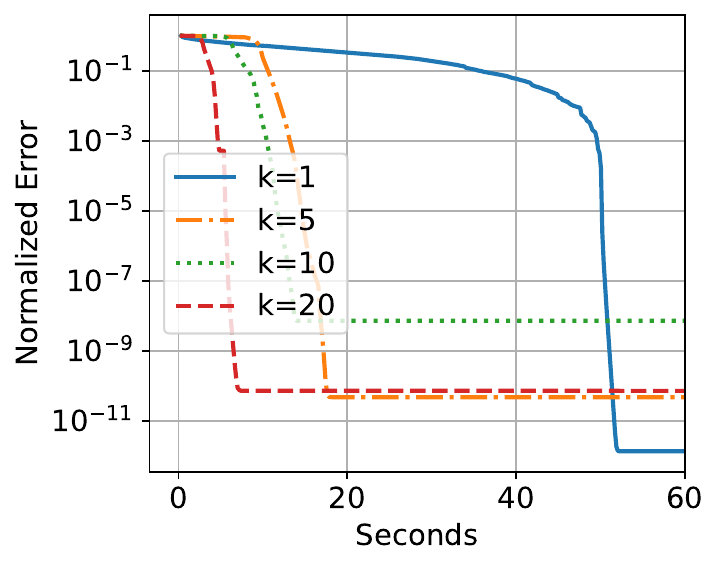}
     & \includegraphics[scale=0.45]{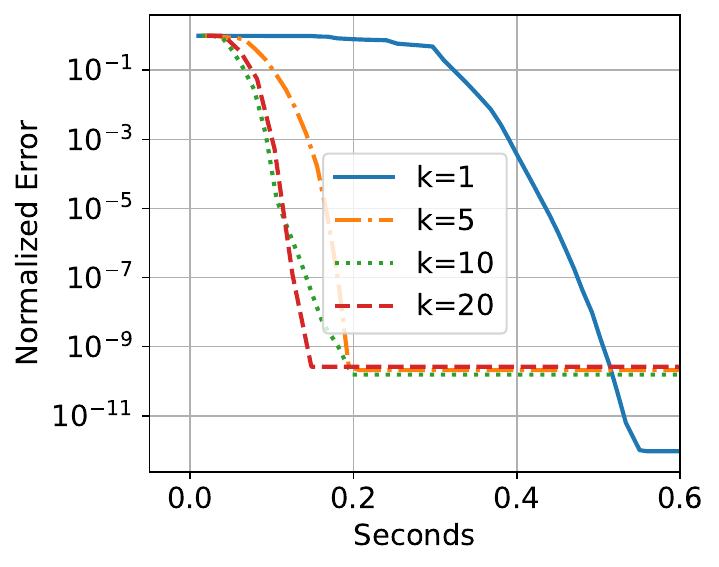} \\
     (a) a9a & (b) w8a &  (c) ijcnn1 \\
     \includegraphics[scale=0.45]{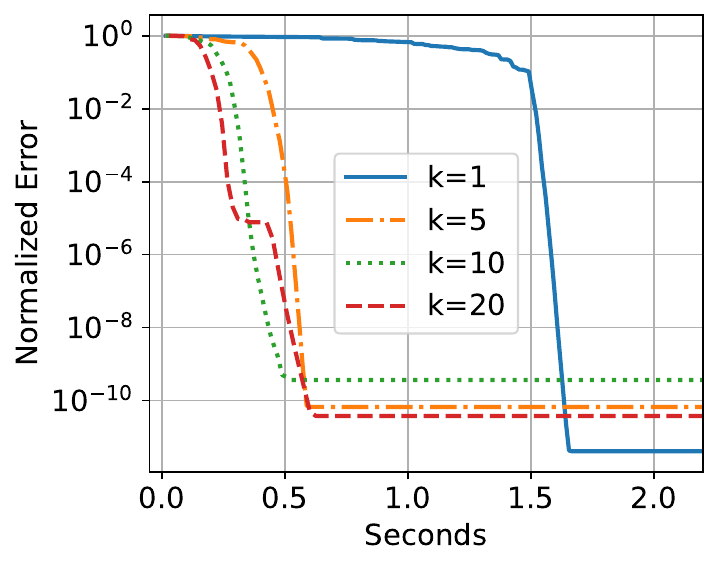}
     & \includegraphics[scale=0.45]{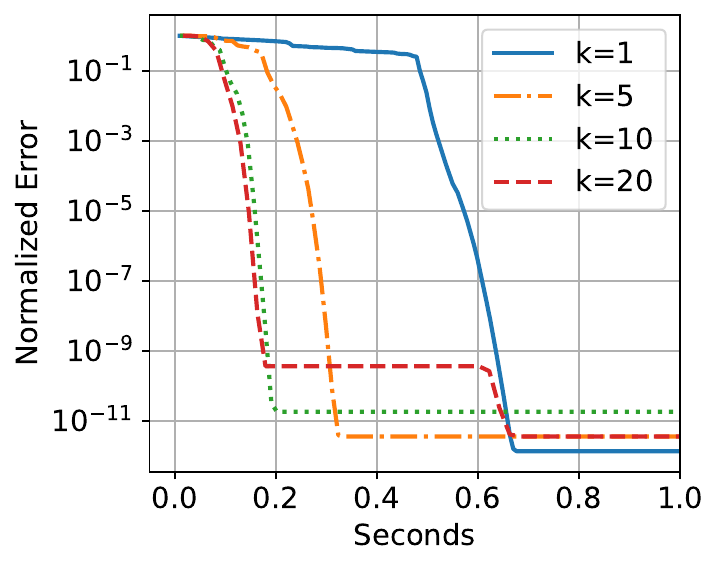}
     & \includegraphics[scale=0.45]{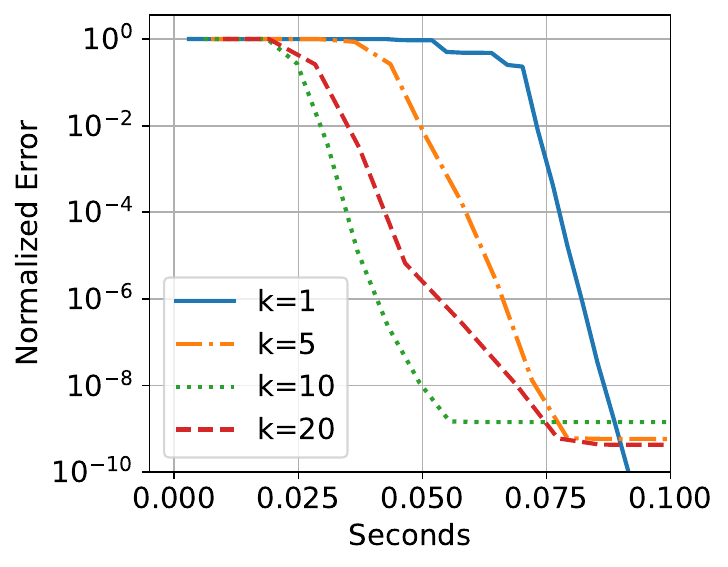} \\   
     (e) mushrooms & (f) phishing &  (g) svmguide3 \\
     \includegraphics[scale=0.45]{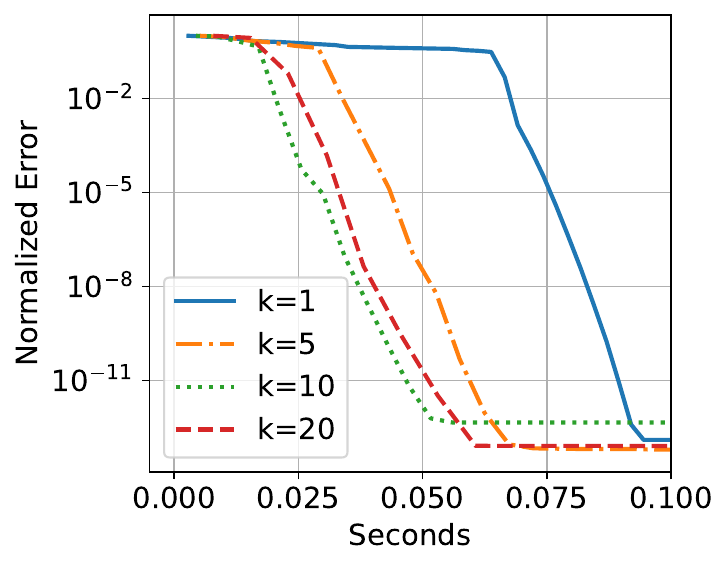}
     & \includegraphics[scale=0.45]{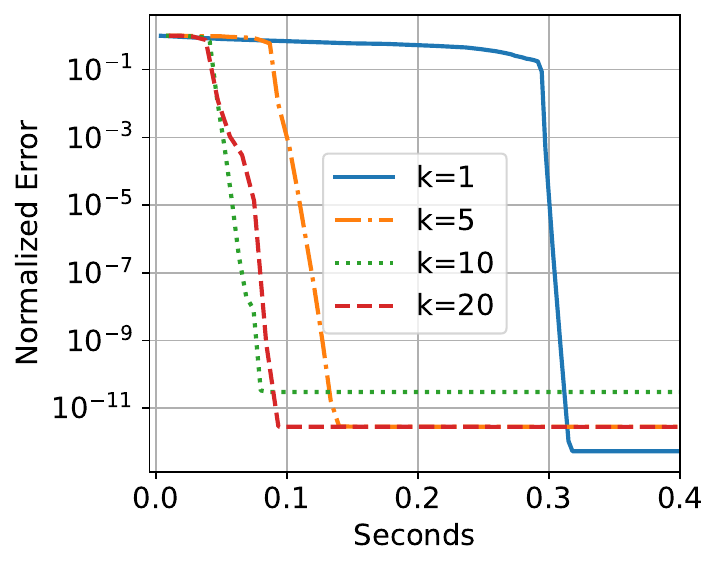}
     & \includegraphics[scale=0.45]{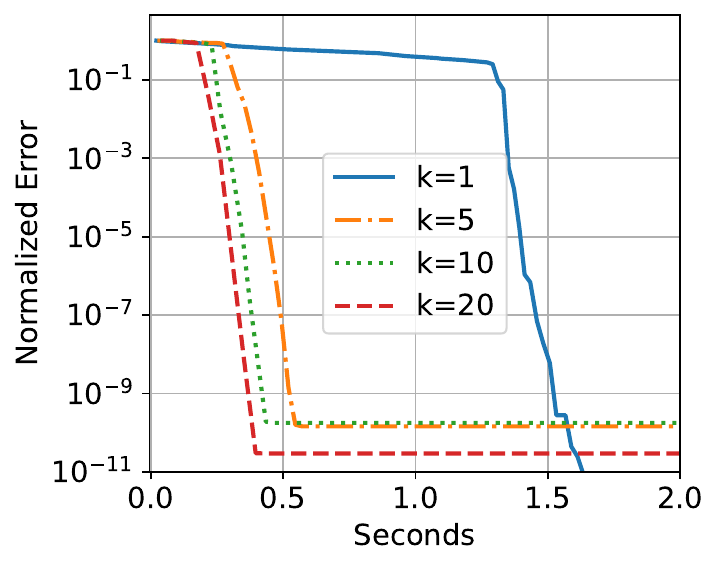}    \\
     (h) german.numer & (i) splice &  (j) covtype \\
\end{tabular}
\caption{Comparison of the LISR-$k$ method with different choices of $k$ for the general function minimization. }
\label{fig:general_compare_time_res}
\end{figure*}
In this subsection, we present the comparison result of the LISR-$k$ method with different choices of $k$ on real-world datasets shown in Table \ref{tab:dataset}.
The comparison result of normalized loss against number of the effective passes can be viewed in Figure \ref{fig:general_compare_res}.
It is not surprising to see that the LISR-$k$ generally converges at a faster rate if $k$ is set larger.
 
Additionally, we plot the comparison result of the normalized loss against seconds in Figure \ref{fig:general_compare_time_res}.
It can be seen that the LISR-$k$ method converges fastest if $k$ is set to 10 or 20. 
Although the per-iteration cost of LISR-$k$ is larger, the overall time to converge is still shorter than LISR-1 thanks to the faster convergence rate.

\end{document}